\documentclass[a4paper,12pt]{amsart}
\usepackage{amssymb}
\usepackage{amscd}
\usepackage{todonotes}
\usepackage{hyperref}
\usepackage{tikz}
\usetikzlibrary{decorations.markings}
\usetikzlibrary{backgrounds,shapes}
\tikzstyle{block}=[ellipse, fill=blue, opacity=0.3]
\tikzstyle{morphism}=[ellipse, fill=green!20!white, draw=black, inner sep=0pt, minimum height=16pt, minimum width=16pt]
\tikzstyle{Xsize}=[minimum size=1.3cm]
\usepackage[backend=bibtex]{biblatex}
\newcommand{\bN}{\mathbb{N}}
\newcommand{\bZ}{\mathbb{Z}}
\newcommand{\bQ}{\mathbb{Q}}

\newcommand{\fS}{\mathfrak{S}}
\newcommand{\fg}{L}
\newcommand{\fso}{\mathfrak{so}}
\newcommand{\fsp}{\mathfrak{sp}}
\newcommand{\fsl}{\mathfrak{sl}}
\newcommand{\Hom}{\mathrm{Hom}}

\newcommand{\cG}{\mathcal{G}}

\newcommand{\End}{\mathrm{End}}

\newcommand{\cS}{\mathcal{S}}

\newcommand{\cD}{\mathcal{D}}
\newcommand{\wD}{\widetilde{\cD}}

\newcommand{\sC}{\mathsf{C}}
\newcommand{\ft}{\mathfrak{t}}
\DeclareMathOperator{\id}{id}

\newtheorem{thm}{Theorem}[section]

\newtheorem{lemma}[thm]{Lemma}
\newtheorem{prop}[thm]{Proposition}

\theoremstyle{definition}
\newtheorem{defn}[thm]{Definition}
\newtheorem{ex}[thm]{Example}

\tikzset{
	vector/.style={draw=red,line width=2pt},
	adjoint/.style={thick, draw=green},
}

\title{Extending and quantising the Vogel plane}
\author{Bruce Westbury}
\date{August 2015}

\begin{document}

\begin{abstract} After introducing the Vogel plane we give the quantisation. Then we extend the Vogel plane to include certain symmetric spaces and give the quantisation of this extension.
\end{abstract}

\maketitle
\tableofcontents

\section{Introduction} The aim of this article is to construct and quantise series of metric Lie algebras. The naive idea of a series of Lie algebras is that it is a family of Lie algebras which depends algebraically on a parameter. The only example of a series of simple Lie algebras in this sense
is the series of simple super Lie algebras $D(2,1;\alpha)$ which depend on the parameter $\alpha$.
A more productive approach is to consider families of categories with the formal properties of a category of representations. The definition of a PROP is given in \cite[\S24. Proper categories]{MR0171826} and for our purposes a PROP is a rigid symmetric linear monoidal category whose monoid of objects is $\bN$. Usually, but not always, a PROP will be finitely presented as a linear monoidal category. Then a more general definition of a series of Lie algebras is that it is a PROP such that the object $[1]$ is a (metric) Lie algebra. In this case we consider the object $[n]$ as the $n$-th tensor power of the adjoint representation. A still more general definition is that there is an idempotent $e\in\End([n])$, for some $n$, such that if we formally adjoin $e$ as a new object then $e$ is a (metric) Lie algebra.

The standard construction of the quantum group associated to a simple Lie algebra starts with the universal enveloping algebra and deforms it as a Hopf algebra. There are two ways to approach the deformation; one is the perturbative approach due to Drinfel$'$d, \cite{MR934283}, and the other is the non-perturbative approach due to Jimbo, \cite{MR797001}. The perturbative approach is
based on deformation quantisation and the non-perturbative approach
defines the Hopf algebra by a finite presentation which is a $q$-analogue of the Serre presentation. These constructions do not generalise to series of Lie algebras as they each depend on a choice of additional structure, such as a Cartan subalgebra; and this structure is not available in our more abstract setting. Instead we quantise the category of representations. The category of finite dimensional representations of a simple Lie algebra is a rigid symmetric monoidal category and has an infinitesimal braiding. The quantisation is then a ribbon category.
We give a more detailed account in \S~\ref{sec:defq}.

In this paper we give examples of non-perturbative quantisations
of series of Lie algebras. This entails giving a $q$-analogue of the finite presentation.
The most well-known examples arise from the classical Lie algebras and we give these examples in \S~\ref{sec:class}. Further examples arise from the
rank two simple Lie algebras and are given in \cite{MR1403861}.

Our interest is in universal examples. The first universal example
is the PROP for metric Lie algebras. The non-perturbative approach is not available, in fact, it is not clear how to formulate this as a problem.
The perturbative approach has been studied and is essentially the theory of Vassiliev (aka. finite type) invariants, \cite{MR1318886}, \cite{MR1913297}. The example that is central to this paper is the universal Lie algebra constructed by Vogel,
\cite{univ}. Heuristically, this is a universal simple Lie algebra. It was introduced in order to show that there are Vassiliev invariants which do not arise from simple Lie superalgebras, \cite[Theorem 8.2]{MR2769234}.

The problem that then arises is to study the quantisation of the universal simple Lie algebra. This problem appears to be inaccessible. In this paper we study a related problem. We substitute a finitely presented category for the universal simple Lie algebra which we call the Vogel plane. Then we quantise by giving a $q$-analogue of the finite presentation.

The quantum dimensions were first given in \cite{MR1815266} and \cite{MR2029689}.
It is remarkable that the first quantum dimension formula also appears independently in \cite{vol} with a different interpretation.

The Vogel plane is a method of parametrising simple Lie algebras by an unordered triple $(\alpha,\beta,\gamma)$, up to scale. Let $\fg$ be the adjoint representation of a simple Lie algebra. Consider the representation $L\otimes L$ and its decomposition into the symmetric square and the exterior square. Then the exterior square contains $L$, using the Lie bracket, and the symmetric square contains the trivial representation using the Killing form. Then the basic observation is that the complement of $L$ in the exterior square is irreducible and the complement of the trivial representation in the symmetric square decomposes into the sum of three irreducible representations (one of which is zero for the exceptional Lie algebras). The Casimir acts as a scalar on each irreducible representation. The Vogel parameters are then determined by the property that these three scalars are $(2t-\alpha,2t-\beta,2t-\gamma)$ where $t=\alpha+\beta+\gamma$.
The parameters $\alpha, \beta, \gamma$ depend on the choice of Casimir and thus are defined up to common multiple and up to permutation. Their values are given in Table~\ref{table:vogel}, where we have chosen the normalisation with $\alpha=-2$ corresponding to the so-called minimal bilinear form, when the square of the length of the maximal root equals 2 and $t=h^\vee$ is the dual Coxeter number.
\begin{table}[h]    
	\begin{tabular}{|c|c|c|c|c|}
		\hline
		Lie algebra    & $\alpha$ & $\beta$ & $\gamma$  & $t=h^\vee$\\   
		\hline    
		$SU(n)$      & $-2$ & 2 & $n $ & $n$\\
		$Spin(n)$    & $-2$ & 4& $n-4 $ & $n-2$\\
		$Sp(2n)$     & $-2$ & 1 & $n+2 $ & $n+1$\\
		$G_{2}  $    & $-2$ & $10/3 $& $8/3$ & $4$ \\
		$F_{4}  $    & $-2$ & $ 5$& $ 6$ & $9$\\
		$E_{6}  $    & $-2$ & $ 6$& $ 8$ & $12$\\
		$E_{7}  $    & $-2$ & $ 8$& $ 12$ & $18$ \\
		$E_{8}  $    & $-2$ & $ 12$& $20$ & $30$\\
		\hline  
	\end{tabular}	
	\caption{Vogel's parameters for simple Lie algebras}\label{table:vogel}
\end{table}
Then there are several formulae which can be written in terms of the Vogel parameters.
Chern-Simons theory is studied from this point of view in
\cite{MR3118578}, \cite{MR3006906}, \cite{MR3050562}.
The volume of a compact group is written in terms of the Vogel
parameters in \cite{vol}. The $R$-matrices in \cite{MR1103907} are written
in terms of the Vogel parameters in \cite{MR1960703}.

Then we show that the finite presentation of the Vogel plane can be modified to include an additional parameter by introducing a second type of edge. We call this the extended Vogel plane. We then quantise by giving a $q$-analogue of the finite presentation.

The problem which motivated the extended Vogel plane was to find an extension of Vogel's parametrisation to irreducible affine Kac-Moody algebras, excluding $D_4^{(3)})$, where each simple Lie algebra is taken to correspond to its untwisted affine Kac-Moody algebra. This then led to a problem on symmetric spaces. Associated to each symmetric space is a graded Lie algebra $L\oplus V$. The exterior square of $V$ contains $L$, by the Lie bracket and the symmetric square of $V$ contains the trivial representation. Then we require that the complement of $L$ in the exterior square of $V$ is irreducible and that the complement of the trivial representation in the symmetric square of $V$ decomposes into the sum of (at most) three irreducible components. This includes simple Lie algebras by taking $L\cong V$. The symmetric spaces with $L\ncong V$ which satisfy this condition are given in \S~\ref{sec:ex}. These can be parametrised using the same method as for the simple Lie algebras but this now requires an additional parameter, $\tau$.
For a simple Lie algebra $\tau=t$ and the parameters in the other cases are given in Table~\ref{table:symm}.
\begin{table}[h]    
	\begin{tabular}{|c|c|c|c|c|c|}
		\hline
		type    & $\tau$ & $\alpha$ & $\beta$ & $\gamma$  & $t$\\   
		\hline 
		AI & $-n+2$ & $-2$ & $4$ & $-n

		$ & $-n+2$ \\
		AII & $n-2$ & $-2$ & $4$ & $n$ & $n+2$ \\
		BDI & $2$ & $-2$ & $\beta$ & $2n-\beta-2$ & $2n-4$ \\ 
		FII & $4$ & $-2$ & $6$ & $10$ & $14$ \\
		EIV & $6$ & $-2$ & $8$ & $12$ & $18$ \\
		EI & $7$ & $-2$ & $3$ & $4$ & $5$ \\
		EV & $10$ & $-2$ & $4$ & $6$ & $8$ \\
		EVIII & $16$ & $-2$ & $6$ & $10$ & $14$ \\
		\hline  
	\end{tabular}
	
	\caption{Parameters for symmetric spaces}\label{table:symm}
\end{table}

The paper is organised as follows. Section~\ref{sec:mon} gives background on monoidal
categories and the diagrammatic notation.
Section~\ref{sec:vogel} gives the construction
of the universal simple Lie algebra. This is 
included for the convenience of the reader.
Section~\ref{sec:plane} gives the presentation of the Vogel plane and section~\ref{sec:plane} gives the $q$-analogue. Section~\ref{sec:planex} gives the presentation of the extended Vogel plane. The examples which inspired this construction are given in section~\ref{sec:ex}. Finally, section~\ref{sec:plane} gives the $q$-analogue of the extended Vogel plane.

\subsection*{Acknowledgement} I would like to thank Dmitriy Rumynin for his patience
and support over many years.

\section{Monoidal categories}\label{sec:mon}
In this section we give a concise account of symmetric monoidal, braided monoidal and
pivotal categories.  Here we only discuss the strict versions of
these concepts; this is for brevity and simplicity and is justified
since the examples constructed from diagrams are strict. The general definition, which involves an associator, is only needed in Section~\ref{sec:defq}.

Monoidal and symmetric monoidal categories were introduced in \cite{MR0170925}
and a coherence theorem is given in \cite{MR593254}. A more recent reference
is \cite{MR1268782}. In this paper we use the diagram calculus for
monoidal and symmetric monoidal categories described in \cite[Sections 1,2]{MR1113284}.
The papers which pioneered this use of diagrams are
\cite{MR0127765}, \cite[Appendix]{MR776784}, \cite{MR2418111}.

\subsection{Monoidal categories}
\begin{defn}
	A \emph{strict monoidal category} consists of the following data:
	\begin{itemize}
		\item a category $\sC$,
		\item a functor $\otimes \colon \sC \times \sC\rightarrow\sC$,
		\item an object $I\in\sC$.
	\end{itemize}
	The functor $\otimes$ is required to be associative. This is the
	condition that
	\begin{equation*}
	\otimes \circ (\otimes \times\id) = \otimes \circ (\id \times \otimes)
	\end{equation*}
	as functors $\sC \times \sC \times \sC\rightarrow\sC$. Equivalently,
	the following diagram commutes:
	\begin{equation*}\begin{CD}
	\sC \times \sC \times \sC @>{\otimes \times\id}>> \sC \times \sC \\
	@V{\id \times \otimes}VV @VV{\otimes}V \\
	\sC \times \sC @>>{\otimes}> \sC
	\end{CD}\end{equation*}
	The object $I$ is a left and right unit for $\otimes$. This means
	that $\id_I \otimes \phi = \phi$ and $\phi \otimes \id_I = \phi$ for
	all morphisms $\phi$.
\end{defn}

Let $\sC$ and $\sC'$ be strict monoidal categories. A functor
$F\colon\sC\rightarrow\sC'$ is \emph{monoidal}
if $F(I)=I'$ and the following diagram commutes
\begin{equation*}\begin{CD}
\sC \times \sC  @>{F\times F}>> \sC' \times \sC' \\
@V{\otimes}VV @VV{\otimes}V \\
\sC @>>{F}> \sC'
\end{CD}\end{equation*}

In this article we will make extensive use of string diagrams for
morphisms in a monoidal
category. A morphism is represented by a graph embedded in a
rectangle with boundary points
on the top and bottom edges. The edges of the graph are directed and
are labelled by objects.
The vertices are labelled by morphisms. A morphism $f\colon
x\rightarrow y$ is drawn as follows:
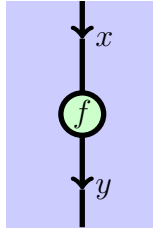
\begin{equation*}
\begin{tikzpicture}[line width=2pt]
\fill[color=blue!20] (-1,-1.5) rectangle (1,1.5);
\draw (0,1.5)[->] -- (0,1) node[anchor=west] {$x$}; \draw[->] (0,1) -- 
(0,-1);
\draw (0,-1) node[anchor=west] {$y$} -- (0,-1.5);
\draw (0,0) node[morphism] {$f$};
\end{tikzpicture}
\end{equation*}

The rules for combining diagrams are that composition is given by
stacking diagrams and the tensor product is given by putting diagrams
side by side.  The corresponding string diagrams are shown in
Figure~\ref{fig:comp}.

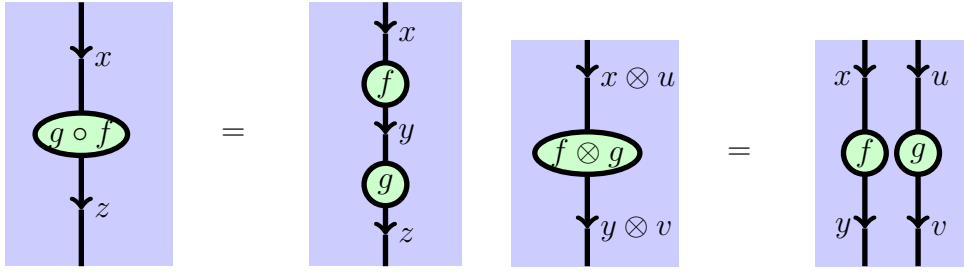
\begin{figure}
	\begin{tikzpicture}[line width=2pt]
	\fill[color=blue!20] (-1,-1.75) rectangle (1.2,1.75);
	\draw (0,1.75)[->] -- (0,1) node[anchor=west] {$x$};
	\draw[->] (0,1) -- (0,-1);
	\draw (0,-1) node[anchor=west] {$z$} -- (0,-1.75);
	\draw (0,0) node[morphism] {$g\circ f$};
	\draw (2,0) node {$=$};
	\fill[color=blue!20] (3,-1.75) rectangle (5,1.75);
	\draw (4,1.75)[->] -- (4,1.333);
	\draw (4,1.333)[->] node[anchor=west] {$x$} -- (4,0);
	\draw (4,0)[->] node[anchor=west] {$y$} -- (4,-1.333);
	\draw (4,-1.333) node[anchor=west] {$z$} -- (4,-1.75);
	\draw (4,0.666) node[morphism] {$f$};
	\draw (4,-0.666) node[morphism] {$g$};
	\end{tikzpicture}
	\hfill
	\begin{tikzpicture}[line width=2pt]
	\fill[color=blue!20] (-1,-1.5) rectangle (1.2,1.5);
	\draw (0,1.5)[->] -- (0,1) node[anchor=west] {$x\otimes u$};
	\draw[->] (0,1) -- (0,-1);
	\draw (0,-1) node[anchor=west] {$y\otimes v$} -- (0,-1.5);
	\draw (0,0) node[morphism] {$f\otimes g$};
	\draw (2,0) node {$=$};
	\fill[color=blue!20] (3,-1.5) rectangle (5,1.5);
	\draw (3.65,1.5)[->] -- (3.65,1) node[anchor=east] {$x$};
	\draw[->] (3.65,1) -- (3.65,-1);
	\draw (3.65,-1) node[anchor=east] {$y$} -- (3.65,-1.5);
	\draw (3.65,0) node[morphism] {$f$};
	\draw (4.35,1.5)[->] -- (4.35,1) node[anchor=west] {$u$};
	\draw[->] (4.35,1) -- (4.35,-1);
	\draw (4.35,-1) node[anchor=west] {$v$} -- (4.35,-1.5);
	\draw (4.35,0) node[morphism] {$g$};
	\end{tikzpicture}
	\caption{Composition and tensor product}\label{fig:comp}
\end{figure}
Rules for simplifying diagrams are that vertices labelled by an
identity morphism can be omitted and edges labelled by the trivial
object, $I$, can be omitted.  The string diagram for the rule that an
identity morphism can be omitted is shown in Figure~\ref{fig:id}.
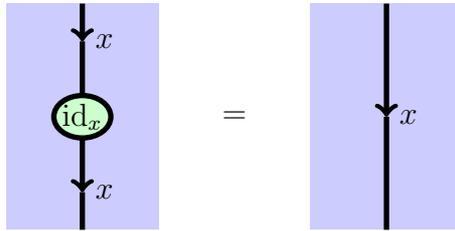
\begin{figure}
	\begin{tikzpicture}[line width=2pt]
	\fill[color=blue!20] (-1,-1.5) rectangle (1,1.5);
	\draw (0,1.5)[->] -- (0,1) node[anchor=west] {$x$}; \draw[->] (0,1) -- 
	(0,-1);
	\draw (0,-1) node[anchor=west] {$x$} -- (0,-1.5);
	\draw (0,0) node[morphism] {$\id_x$};
	\draw (2,0) node {$=$};
	\fill[color=blue!20] (3,-1.5) rectangle (5,1.5);
	\draw (4,1.5)[->] -- (4,0); \draw (4,0) node[anchor=west] {$x$} -- (4,-1.5);
	\end{tikzpicture}
	\caption{Identity morphism}\label{fig:id}
\end{figure}

\subsection{Braided monoidal categories}
Braided monidal categories were introduced in \cite{MR1250465}.
\begin{defn} A \emph{strict braided monoidal category} is a strict
	monoidal category $\sC$ together with
	natural isomorphisms $\sigma(x,y)\colon x\otimes y\rightarrow
	y\otimes x$ for all objects $x,y$.
	These are required to satisfy the following conditions:
	\begin{align}
	\sigma(v,x) \circ (f\otimes \id_x) &= (\id_x\otimes f) \circ
	\sigma(u,x) \label{a:slide1} \\
	\sigma(x,v) \circ (\id_x\otimes f) &= (f\otimes\id_x) \circ
	\sigma(x,u) \label{a:slide2}
	\end{align}
	for all $f\colon u\rightarrow v$ and all $x$, and
	\begin{align}
	\sigma(x,y\otimes
	z)&=\id_y\otimes\sigma(x,z)\,\circ\,\sigma(x,y)\otimes\id_z
	\label{a:tensor1} \\
	\alpha(x\otimes
	y,z)&=\alpha(x,z)\otimes\id_y\,\circ\,\id_x\otimes\alpha(y, z)
	\label{a:tensor2}
	\end{align}
	for all $x,y,z$.
	
	A \emph{braided monoidal functor} is a functor compatible with this
	structure.
\end{defn}

The string diagram for $\sigma(x,y)$ is as follows:
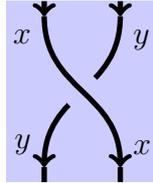
\begin{equation*}
\begin{tikzpicture}[line width=2pt]
\fill[color=blue!20] (-1,-1.2) rectangle (1,1.2);
\draw (-0.5,-1) -- (-0.5,-1.2);
\draw (0.5,-1) -- (0.5,-1.2);
\draw (-0.5,1.2)[->] -- (-0.5,1);
\draw (0.5,1.2)[->] -- (0.5,1);
\draw (0.5,1)[->] node[anchor=north west] {$y$}  .. controls (0.5,0) and
(-0.5,0) .. (-0.5,-1) node[anchor=south east] {$y$};
\fill[color=blue!20] (0,0) circle(0.25);
\draw (-0.5,1)[->] node[anchor=north east] {$x$} .. controls (-0.5,0) and
(0.5,0) .. (0.5,-1)node[anchor=south west] {$x$};
\end{tikzpicture}
\end{equation*}
The string diagrams for
conditions~\eqref{a:slide1} and~\eqref{a:slide2} are shown in
Figure~\ref{fig:rels-slide}, and those for
conditions~\eqref{a:tensor1} and~\eqref{a:tensor2} are shown in
Figure~\ref{fig:rels-tensor}.
\begin{figure}
	\begin{tikzpicture}[line width=2pt]
	\fill[color=blue!20] (-1,-1.2) rectangle (1,3.2);
	\draw (-0.5,-1) node[anchor=south east] {$x$} -- (-0.5,-1.2);
	\draw (0.5,-1) node[anchor=south west] {$v$} -- (0.5,-1.2);
	\draw (-0.5,3.2)[->] -- (-0.5,1);
	\draw (0.5,3.2)[->] -- (0.5,1);
	\draw (-0.5,3.2)[->] -- (-0.5,3) node[anchor=north east] {$u$};
	\draw (0.5,3.2)[->] -- (0.5,3)  node[anchor=north west] {$x$};
	\draw (0.5,1)[->]  node[anchor=west] {$x$} .. controls (0.5,0) and
	(-0.5,0) .. (-0.5,-1);
	\fill[color=blue!20] (0,0) circle(0.25);
	\draw (-0.5,1)[->] node[anchor=east] {$v$}  .. controls (-0.5,0) and
		(0.5,0) .. (0.5,-1) ;
	\draw (-0.5,2) node[morphism] {$f$};
	\draw (1.5,1) node {$=$};
	\fill[color=blue!20] (2,-1.2) rectangle (4,3.2);
	\draw (2.5,3.2)[->] -- (2.5,3) node[anchor=north east] {$u$};
	\draw (3.5,3.2)[->] -- (3.5,3)  node[anchor=north west] {$x$};
	\draw (2.5,1) node[anchor=  east] {$x$} -- (2.5,-1.2);
	\draw (2.5,-1) node[anchor= south east] {$x$};
	\draw (3.5,1) node[anchor=  west] {$u$} -- (3.5,-1.2);
	\draw (3.5,-1) node[anchor= south west] {$v$};
	\draw (3.5,3)[->]  .. controls (3.5,2) and (2.5,2) .. (2.5,1) ;
	\fill[color=blue!20] (3,2) circle(0.25);
	\draw (2.5,3)[->]  .. controls (2.5,2) and (3.5,2) .. (3.5,1) ;
	\draw (3.5,0) node[morphism] {$f$};
	\end{tikzpicture}
	\hfill
	\begin{tikzpicture}[line width=2pt]
	\fill[color=blue!20] (-1,-1.2) rectangle (1,3.2);
	\draw (-0.5,-1) node[anchor=south east] {$v$} -- (-0.5,-1.2);
	\draw (0.5,-1) node[anchor=south west] {$x$} -- (0.5,-1.2);
	\draw (-0.5,3.2)[->] -- (-0.5,1) node[anchor=  east] {$x$};
	\draw (0.5,3.2)[->] -- (0.5,1) node[anchor=  west] {$v$};
	\draw (-0.5,3.2)[->] -- (-0.5,3) node[anchor=north east] {$x$};
	\draw (0.5,3.2)[->] -- (0.5,3)  node[anchor=north west] {$u$};
	\draw (0.5,1)[->]  .. controls (0.5,0) and (-0.5,0) .. (-0.5,-1) ;
	\fill[color=blue!20] (0,0) circle(0.25);
	\draw (-0.5,1)[->]   .. controls (-0.5,0) and (0.5,0) .. (0.5,-1) ;
	\draw (0.5,2) node[morphism] {$f$};
	\draw (1.5,1) node {$=$};
	
	\fill[color=blue!20] (2,-1.2) rectangle (4,3.2);
	\draw (2.5,3.2)[->] -- (2.5,3) node[anchor=north east] {$x$};
	\draw (3.5,3.2)[->] -- (3.5,3)  node[anchor=north west] {$u$};
	\draw (2.5,1) -- (2.5,-1.2);
	\draw (3.5,1) -- (3.5,-1.2);
	\draw (2.5,-1) node[anchor = south east] {$v$};
	\draw (3.5,-1) node[anchor = south west] {$x$};
	\draw (3.5,3)[->]  .. controls (3.5,2) and (2.5,2) .. (2.5,1)
	node[anchor=  east] {$u$};
	\fill[color=blue!20] (3,2) circle(0.25);
	\draw (2.5,3)[->]  .. controls (2.5,2) and (3.5,2) .. (3.5,1)
	node[anchor=  west] {$x$};
	\draw (2.5,0) node[morphism] {$f$};
	\end{tikzpicture}
	\caption{Relations~\eqref{a:slide1} and~\eqref{a:slide2}}\label{fig:rels-slide}
\end{figure}
\begin{figure}
	\begin{tikzpicture}[line width=2pt]
	\fill[color=blue!20] (-1.75,-1.2) rectangle (1.75,1.2);
	\draw (-0.5,-1) -- (-0.5,-1.2);
	\draw (0.5,-1) -- (0.5,-1.2);
	\draw (-0.5,1.2)[->] -- (-0.5,1);
	\draw (0.5,1.2)[->] -- (0.5,1);
	\draw (0.5,1)[->] node[anchor=north west] {$z$}  .. controls (0.5,0)
	and (-0.5,0) .. (-0.5,-1) node[anchor=south east] {$z$};
	\fill[color=blue!20] (0,0) circle(0.25);
	\draw (-0.5,1)[->] node[anchor=north east] {$x\otimes y$}
	.. controls (-0.5,0) and (0.5,0) .. (0.5,-1) node[anchor=south west]
	{$x\otimes y$};
	\draw (2.1,0) node {$=$};
	\begin{scope}[xshift=3.5cm]
	\fill[color=blue!20] (-1,-1.2) rectangle (2,1.2);
	\draw (-0.5,-1) -- (-0.5,-1.2);
	\draw (0.5,-1) -- (0.5,-1.2);
	\draw (1.5,-1) -- (1.5,-1.2);
	\draw (-0.5,1.2)[->] -- (-0.5,1);
	\draw (0.5,1.2)[->] -- (0.5,1);
	\draw (1.5,1.2)[->] -- (1.5,1);
	\draw (1.5,1)[->] node[anchor=north west] {$z$}  .. controls (1.5,0)
	and (-0.5,0) .. (-0.5,-1) node[anchor=south east] {$z$};
	\fill[color=blue!20] (0.2,-0.2) circle(0.2);
	\fill[color=blue!20] (0.8,0.2) circle(0.2);
	\draw (-0.5,1)[->] node[anchor=north east] {$x$}  .. controls
	(-0.5,0) and (0.5,0) .. (0.5,-1) node[anchor=south west] {$x$};
	\draw (0.5,1)[->] node[anchor=north west] {$y$}  .. controls (0.5,0)
	and (1.5,0) .. (1.5,-1) node[anchor=south west] {$y$};	
	\end{scope}
	\end{tikzpicture}
	
	\vspace{1cm}
	
	\begin{tikzpicture}[line width=2pt]
	\fill[color=blue!20] (-1.75,-1.2) rectangle (1.75,1.2);
	\draw (-0.5,-1) -- (-0.5,-1.2);
	\draw (0.5,-1) -- (0.5,-1.2);
	\draw (-0.5,1.2)[->] -- (-0.5,1);
	\draw (0.5,1.2)[->] -- (0.5,1);
	\draw (0.5,1)[->] node[anchor=north west] {$y\otimes z$}  .. controls
	(0.5,0) and (-0.5,0) .. (-0.5,-1) node[anchor=south east] {$y\otimes
		z$};
	\fill[color=blue!20] (0,0) circle(0.25);
	\draw (-0.5,1)[->] node[anchor=north east] {$x$}  .. controls
	(-0.5,0) and (0.5,0) .. (0.5,-1) node[anchor=south west] {$x$};
	\draw (2.1,0) node {$=$};
	\begin{scope}[xshift=3.5cm]
	\fill[color=blue!20] (-1,-1.2) rectangle (2,1.2);
	\draw (-0.5,-1) -- (-0.5,-1.2);
	\draw (0.5,-1) -- (0.5,-1.2);
	\draw (1.5,-1) -- (1.5,-1.2);
	\draw (-0.5,1.2)[->] -- (-0.5,1);
	\draw (0.5,1.2)[->] -- (0.5,1);
	\draw (1.5,1.2)[->] -- (1.5,1);
	\draw (-0.5,1)[->] node[anchor=north east] {$x$}  .. controls
	(-0.5,0) and (1.5,0) .. (1.5,-1) node[anchor=south west] {$x$};
	\fill[color=blue!20] (0.8,-0.2) circle(0.2);
	\fill[color=blue!20] (0.2,0.2) circle(0.2);
	\draw (0.5,1)[->] node[anchor=north east] {$y$}  .. controls (0.5,0)
	and (-0.5,0) .. (-0.5,-1) node[anchor=south west] {$y$};	
	\draw (1.5,1)[->] node[anchor=north west] {$z$}  .. controls (1.5,0)
	and (0.5,0) .. (0.5,-1) node[anchor=south east] {$z$};
	\end{scope}
	\end{tikzpicture}
	\caption{Relations~\eqref{a:tensor1} and~\eqref{a:tensor2}}\label{fig:rels-tensor}
\end{figure}
Note that applying the relations~\eqref{a:slide1},~\eqref{a:slide2}
to $f=\sigma$
give the braid relations for the morphisms $\sigma$.

A \emph{symmetric} monoidal category is a braided monoidal category
for which the braiding satisfies the additional condition that
	\begin{equation}\label{a:square}
	\alpha(x,y)\circ\alpha(y,x) = \id_{y\otimes x}
	\end{equation}
	for all $x,y$.
	
In this case the string diagram for $\sigma(x,y)$ is given in
Figure~\ref{fig:flip}.
\begin{figure}
\begin{equation*}
\begin{tikzpicture}[line width=2pt]
\fill[color=blue!20] (-1,-1.2) rectangle (1,1.2);
\draw (-0.5,-1) -- (-0.5,-1.2);
\draw (0.5,-1) -- (0.5,-1.2);
\draw (-0.5,1.2)[->] -- (-0.5,1);
\draw (0.5,1.2)[->] -- (0.5,1);
\draw (0.5,1)[->] node[anchor=north west] {$y$}  .. controls (0.5,0) and (-0.5,0) .. (-0.5,-1) node[anchor=south east] {$y$};
\draw (-0.5,1)[->] node[anchor=north east] {$x$}  .. controls (-0.5,0) and (0.5,0) .. (0.5,-1) node[anchor=south west] {$x$};
\end{tikzpicture}
\end{equation*}
\caption{The symmetric flip}\label{fig:flip}
\end{figure}
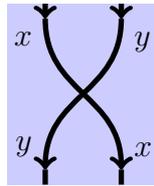

\begin{figure}
	\begin{tikzpicture}[line width=2pt]
	\fill[color=blue!20] (-1,-1.2) rectangle (1,3.2);
	\draw (-0.5,-1) -- (-0.5,-1.2);
	\draw (0.5,-1) -- (0.5,-1.2);
	\draw (-0.5,3.2)[->] -- (-0.5,3) node[anchor=north east] {$x$};
	\draw (0.5,3.2)[->] -- (0.5,3)  node[anchor=north west] {$y$};
	\draw (-0.5,1)[->]   .. controls (-0.5,0) and (0.5,0) .. (0.5,-1)
	node[anchor=south west] {$y$};
	\draw (0.5,1)[->]  .. controls (0.5,0) and (-0.5,0) .. (-0.5,-1)
	node[anchor=south east] {$x$};
	\draw (-0.5,3)[->]  .. controls (-0.5,2) and (0.5,2) .. (0.5,1)
	node[anchor= west] {$x$};
	\draw (0.5,3)[->]  .. controls (0.5,2) and (-0.5,2) .. (-0.5,1)
	node[anchor= east] {$y$};
	\draw (1.5,1) node {$=$};
	\fill[color=blue!20] (2,-1.2) rectangle (4,3.2);
	\draw (2.5,3.2)[->] -- (2.5,1);
	\draw (2.5,1) node[anchor= east] {$x$} -- (2.5,-1.2);
	\draw (3.5,3.2)[->] -- (3.5,1);
	\draw (3.5,1) node[anchor= west] {$y$} -- (3.5,-1.2);
	\end{tikzpicture}
	\caption{Symmetry}\label{fig:sym}
\end{figure}
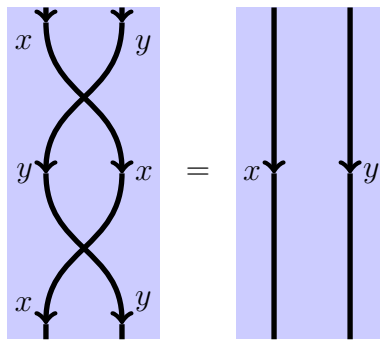
The string diagram for condition~\eqref{a:square} is shown in
Figure~\ref{fig:sym}.
\clearpage
\subsection{Duality}
\begin{defn}
	Let $x$ and $y$ be objects in a monoidal category. Then $x$ is a
	\emph{left dual}
	of $y$ and $y$ is a \emph{right dual} of $x$ means that we are given
	evaluation and
	coevaluation morphisms
	$I\rightarrow x\otimes y$ and $y\otimes x\rightarrow I$ such that
	both of the
	following composites are identity morphisms
	\begin{align*}
	x\rightarrow I\otimes x\rightarrow x\otimes y\otimes x\rightarrow
	x\otimes I\rightarrow x \label{s:trl} \\
	y\rightarrow y\otimes I\rightarrow y\otimes x\otimes y\rightarrow
	I\otimes y\rightarrow y. 
	\end{align*}
	The object $x$ is \emph{dual} to $y$ if it is both a left dual and a
	right dual.
\end{defn}

In the string diagrams we adopt the convention depicted in
Figure~\ref{fig:dual}.
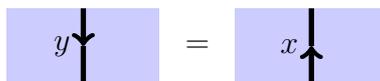
\begin{figure}
	\begin{tikzpicture}[line width=2pt]
	\fill[color=blue!20] (-1,-0.5) rectangle (1,0.5);
	\draw (0,0.5)[->] -- (0,0) node[anchor=east] {$y$};
	\draw (0,0) -- (0,-0.5);
	\draw (1.5,0) node {$=$};
	\fill[color=blue!20] (2,-0.5) rectangle (4,0.5);
	\draw (3,-0.5)[->] -- (3,0) node[anchor=east] {$x$};
	\draw (3,0) -- (3,0.5);
	\end{tikzpicture}
	\caption{Convention for $x$ being a dual of $y$}\label{fig:dual}
\end{figure}
Using this convention and omitting the edge labelled $I$ the string diagrams
for the evaluation and coevaluation morphisms are shown in Figure~{fig:ev}
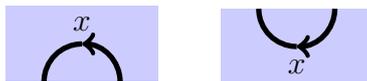
\begin{figure}
\begin{equation*}
\begin{tikzpicture}[line width=2pt]
\fill[color=blue!20] (-1,-1) rectangle (1,0);
\draw (0.5,-1)[->] arc (0:90:0.5);
\draw (0,-0.5) node[anchor=south] {$x$} arc (90:180:0.5);
\end{tikzpicture}
\qquad
\begin{tikzpicture}[line width=2pt]
\fill[color=blue!20] (-1,-1) rectangle (1,0);
\draw (0.5,0)[->] arc (0:-90:0.5);
\draw (0,-0.5) node[anchor=north] {$x$} arc (-90:-180:0.5);
\end{tikzpicture}
\end{equation*}
\caption{Evaluation and coevaluation}\label{fig:ev}
\end{figure}
The string diagrams for the conditions on these morphisms are shown
in Figure~\ref{fig:cupcap}.
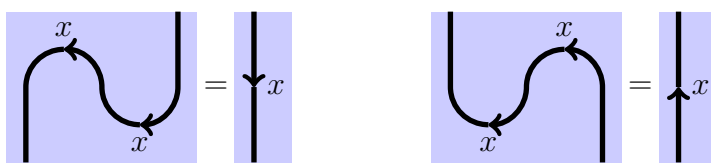
\begin{figure}
	\begin{tikzpicture}[line width=2pt]
	\fill[color=blue!20] (-1.25,-1) rectangle (1.25,1);
	\draw[->] (1,1) -- (1,0) arc(0:-90:0.5);
	\draw[->] (0.5,-0.5) node[anchor=north] {$x$} arc(-90:-180:0.5)
	arc(0:90:0.5);
	\draw (-0.5,0.5) node[anchor=south] {$x$} arc(90:180:0.5) -- (-1,-1);
	\draw (1.5,0) node {$=$};
	\fill[color=blue!20] (1.75,-1) rectangle (2.5,1);
	\draw[->] (2,1) -- (2,0);
	\draw (2,0) node[anchor=west] {$x$} -- (2,-1);
	\end{tikzpicture}
	\qquad\qquad
	\begin{tikzpicture}[line width=2pt]
	\fill[color=blue!20] (-1.25,-1) rectangle (1.25,1);
	\draw[->] (1,-1) -- (1,0) arc(0:90:0.5);
	\draw[->] (0.5,0.5) node[anchor=south] {$x$} arc(90:180:0.5) arc(0:-90:0.5);
	\draw (-0.5,-0.5) node[anchor=north] {$x$} arc(-90:-180:0.5) -- (-1,1);
	\draw (1.5,0) node {$=$};
	\fill[color=blue!20] (1.75,-1) rectangle (2.5,1);
	\draw[->] (2,-1) -- (2,0);
	\draw (2,0) node[anchor=west] {$x$} -- (2,1);
	\end{tikzpicture}
	\caption{Relations for duals}\label{fig:cupcap}
\end{figure}

It follows that there are natural homomorphisms
\begin{equation*}
\Hom(y\otimes u,v)\rightarrow \Hom(u,x\otimes v)
\end{equation*}
sending $\phi\colon y\otimes u\rightarrow v$ to the composite
\begin{equation*}
u\rightarrow I\otimes u\rightarrow x\otimes y\otimes u\rightarrow
x\otimes v.
\end{equation*}
Similarly there are natural homomorphisms
\begin{equation*}
\Hom(u,x\otimes v)\rightarrow \Hom(y\otimes u,v).
\end{equation*}
sending $\phi\colon u\rightarrow x\otimes v$ to the composite
\begin{equation*}
y\otimes u\rightarrow y\otimes x\otimes v\rightarrow I\otimes
v\rightarrow  v.
\end{equation*}
It follows from the relations in Figure~\ref{fig:cupcap} that these
are inverse isomorphisms.

\begin{ex} Take a category $\sC$ and consider the category whose objects are
	functors $F\colon \sC \rightarrow \sC$ and whose morphisms are natural
	transformations. This is a monoidal category with $\otimes$ given by
	composition.
	Then $F$ is left dual to $G$ in this monoidal category is equivalent to
	$F$ is left adjoint to $G$.
\end{ex}

\begin{ex} Let $K$ be a field. The category of vector spaces over $K$ is
	a symmetric monoidal category. A vector space has a dual if and only if
	it is finite dimensional. More generally, Let $K$ be a commutative ring.
	The category of $K$-modules is a symmetric monoidal category.
	A $K$-module has a dual if and only if it is finitely generated and
	projective.
\end{ex}

\begin{defn}
	A \emph{strict pivotal category} is a strict monoidal category with
	an antimonoidal antiinvolution $\ast$ such that $x^\ast$ is a dual of
	$x$ for all objects $x$.
\end{defn}

\section{Deformation quantisation}\label{sec:defq}
This section is based on the unpublished lecture notes \cite{they}
and is also discussed in \cite{MR1664995}.

\begin{defn} A symmetric monoidal category is \emph{infinitesimally-braided} if there
is given a natural transformation $t\colon \otimes\rightarrow\otimes$ that commutes with the symmetry and which satisfies the four-term relation.
\end{defn}

\begin{defn} For $n\ge 2$, the Kohno-Drinfeld Lie algebra $\ft_n$ is the finitely
	presented Lie algebra with generators $t_{ij}=t_{ji}$ for $i\ne j$ and $1\le i,j\le n$
	and the defining relations are
	\begin{align*}
		[t_{ij},t_{kl}] &=0 &\text{if $i,j,k,l$ are distinct}\\
		[t_{ij},t_{ik}+t_{jk}] &=0 &\text{if $i,j,k$ are distinct}
	\end{align*}
\end{defn}

This Lie algebra is graded by assigning degree 1 to each generator.
\begin{equation*}
	\dim(\ft_n^1)=\frac{n(n-1)}2 \qquad \dim(\ft_n^2)=\frac{n(n-1)(n-2)}2
\end{equation*}

\begin{defn} A \emph{Casimir element} in a Lie algebra, $\fg$,
is an invariant element $C\in S^2(\fg)$.
\end{defn}

\begin{ex}\label{ex:fg} Let $\fg$ be a Lie algebra with a Casimir element $C$. Then the category
	of finite dimensional $\fg$-modules is an infinitesimally-braided symmetric monoidal category.
\end{ex}

\begin{prop} Let $(\sC,\otimes,I,\sigma)$ be a braided monoidal category over $K[[\hbar]]$.
Then $(\sC/\hbar,[\otimes],[I],[\sigma])$ is a braided monoidal category. Assume that this
is, in fact, symmetric monoidal. Then it is infinitesimally-braided with 
$t=[(s^2-1)/\hbar]$.
\end{prop}
In this case we say the braided monoidal category is a deformation quantisation of the
infinitesimally-braided symmetric monoidal category.

Conversely, given an infinitesimally-braided symmetric monoidal category and a
Drinfeld associator, $\Phi$, there is a construction of a 
deformation quantisation. The category is $\sC[[\hbar]]\cong \sC\otimes_K K[[\hbar]]$.
The ring  $K[[\hbar]]$ is a local ring and the quotient by the maximal ideal is a homomorphism $ K[[\hbar]]\rightarrow K$. This gives, by specialisation, a functor $[\;]\colon \sC[[\hbar]]\rightarrow \sC$.
The braiding is $\sigma = s\circ \exp(\hbar\, C/2)$. The monoidal structure is modified
by taking the associator to be 
\[ \alpha_{XYZ} = \Phi(\hbar\, C_{XY},\hbar\, C_{YZ}) \]

The choice of associator can be 
restricted by the requirement that $[\alpha]$ is the original
associator so that $[\;]$ is a monoidal functor.

The following Theorem, discussed in \cite{MR1187296}, shows that there is no additional issue with duals.
\begin{thm} Any deformation quantisation of an infinitesimally-braided rigid symmetric monoidal category is a ribbon category.
\end{thm}

\begin{ex} Let $\fg$ be a semisimple Lie algebra with $U_q(\fg)$ the associated
	quantised enveloping algebra. Then construct a braided monoidal category
	over $\bQ[[\hbar]]$ from the category of finite dimensional representations
	of $U_q(\fg)$ by substituting $q=\exp(\hbar)$ and expanding into formal power series.
	Then this is a deformation quantisation of Example~\ref{ex:fg} with
	Casimir constructed from the Killing form.
\end{ex}

The result that these two approaches agree is given in \cite{MR1321294}.

\subsection{Classical Lie algebras}\label{sec:class}
In this section we give two examples of a series of Lie algebras constructed from the classical groups and their quantisation.

In both examples we use the ring defined as the quotient of $\bZ[q,z,(zq)^{-1},\delta]$ by the relation
\begin{equation}
z-z^{-1}=(q-q^{-1})\delta
\end{equation}

The first example we take the directed Brauer category which is also the cobordism category of oriented one-manifolds. This is the free rigid symmetric monoidal on an object. It follows that it is also an intial object in the category of PROPs.

The object
\begin{tikzpicture}[red,line width=2pt,scale=0.5]
\fill[blue!30] (0,0) rectangle (1.5,1);
\draw[decoration={markings,mark=at position 0.5 with {\arrow[ultra thick]{>}}},postaction={decorate}] (0.5,1) -- (0.5,0);
\draw[decoration={markings,mark=at position 0.5 with {\arrow[ultra thick]{>}}},postaction={decorate}] (1,0) -- (1,1);
\end{tikzpicture}
is an associative algebra with multiplication map
\begin{equation}
\begin{tikzpicture}[red,line width=2pt]
\fill[blue!30] (0,0) rectangle (2.5,1);
\draw[decoration={markings,mark=at position 0.5 with {\arrow[ultra thick]{>}}},postaction={decorate}] (0.5,1) to [out=270,in=90] (1,0);
\draw[decoration={markings,mark=at position 0.5 with {\arrow[ultra thick]{>}}},postaction={decorate}] (1.5,0) to [out=90,in=270] (2,1);
\draw[decoration={markings,mark=at position 0.5 with {\arrow[ultra thick]{>}}},postaction={decorate}] (1.5,1) arc (360:180:0.25);
\end{tikzpicture}
\end{equation}

Taking the commutator gives a Lie algebra. This is a metric Lie algebra using
\begin{equation}
\begin{tikzpicture}[red,line width=2pt]
\fill[blue!30] (0,0) rectangle (2.5,1);
\draw[decoration={markings,mark=at position 0.5 with {\arrow[ultra thick]{>}}},postaction={decorate}] (1.5,1) arc (360:180:0.25);
\draw[decoration={markings,mark=at position 0.5 with {\arrow[ultra thick]{>}}},postaction={decorate}] (0.5,1) arc (180:360:0.75);
\end{tikzpicture}
\qquad
\begin{tikzpicture}[vector]
\fill[blue!30] (0,0) rectangle (2.5,1);
\draw[decoration={markings,mark=at position 0.5 with {\arrow[ultra thick]{>}}},postaction={decorate}] (2,0) arc (0:180:0.75);
\draw[decoration={markings,mark=at position 0.5 with {\arrow[ultra thick]{>}}},postaction={decorate}] (1,0) arc (180:0:0.25);
\end{tikzpicture}
\end{equation}

The quantisation is the quotient of the category of oriented tangles by the HOMFLY skein relation given in Figure~\ref{fig:hf} and the basic relations given in Figure~\ref{fig:hb}.

\begin{equation*}\begin{tikzpicture}[vector]
\draw[decoration={markings,mark=at position 1 with {\arrow[ultra thick]{>}}},postaction={decorate}] (0,0) circle(0.5);
\path (1.5,0) node[black] {$\displaystyle =\delta$};
\begin{pgfonlayer}{background}
\fill[fill=blue!30] (0,0) circle (1cm);
\fill[fill=blue!30] (3,0) circle (1cm);
\end{pgfonlayer}
\end{tikzpicture}\qquad
\begin{tikzpicture}[vector]
\draw[decoration={markings,mark=at position 1 with {\arrow[ultra thick]{<}}},postaction={decorate}] (0,0) circle(0.5);
\path (1.5,0) node[black] {$\displaystyle =\delta$};
\begin{pgfonlayer}{background}
\fill[fill=blue!30] (0,0) circle (1cm);
\fill[fill=blue!30] (3,0) circle (1cm);
\end{pgfonlayer}
\end{tikzpicture}
\end{equation*}

\begin{figure}
\begin{equation*}
\begin{tikzpicture}[vector]
\fill[blue!30] (0,0) circle (1cm);
\draw[decoration={markings,mark=at position 0.3 with {\arrow[ultra thick]{>}}},postaction={decorate}] (45:1) -- (225:1);
\fill[blue!30] (0,0) circle (0.25cm);
\draw[decoration={markings,mark=at position 0.3 with {\arrow[ultra thick]{>}}},postaction={decorate}] (135:1) -- (315:1);
\end{tikzpicture}
\raisebox{0.9cm}{\;$-$\;}
\begin{tikzpicture}[vector]
\fill[blue!30] (0,0) circle (1cm);
\draw[decoration={markings,mark=at position 0.3 with {\arrow[ultra thick]{>}}},postaction={decorate}] (135:1) -- (315:1);
\fill[blue!30] (0,0) circle (0.25cm);
\draw[decoration={markings,mark=at position 0.3 with {\arrow[ultra thick]{>}}},postaction={decorate}] (45:1) -- (225:1);
\end{tikzpicture}
\raisebox{0.9cm}{\;$=(q-q^{-1})$\;}
\begin{tikzpicture}[red,line width=2pt]
\fill[blue!30] (0,0) circle (1cm);
\draw[decoration={markings,mark=at position 0.5 with {\arrow[ultra thick]{>}}},postaction={decorate}] (45:1) .. controls (45:0.5) and (315:0.5) .. (315:1);
\draw[decoration={markings,mark=at position 0.5 with {\arrow[ultra thick]{>}}},postaction={decorate}] (135:1) .. controls (135:0.5) and (225:0.5) .. (225:1);
\end{tikzpicture}
\end{equation*}
\caption{Skein relation}\label{fig:hf}
\end{figure}

\begin{figure}
\begin{equation*}
\begin{tikzpicture}[vector]
\fill[fill=blue!30] (0,0) circle (1cm);
\draw (225:1) -- (45:0.25);
\draw (45:0.25) to [out=45,in=90] (0.75,0);
\fill[fill=blue!30] (0,0) circle (0.2cm);
\draw[decoration={markings,mark=at position 0.5 with {\arrow[ultra thick]{>}}},postaction={decorate}] (135:1) -- (315:0.25);
\draw (315:0.25) to [out=-45,in=270] (0.75,0);
\end{tikzpicture}
\raisebox{0.9cm}{\,$=z$\,}
\begin{tikzpicture}[vector]
\fill[fill=blue!30] (0,0) circle (1cm);
\draw[decoration={markings,mark=at position 1 with {\arrow[ultra thick]{>}}},postaction={decorate}] (135:1) .. controls (135:0.5) and (-0.25,0.25) .. (-0.25,0);
\draw (-0.25,0) .. controls (-0.25,-0.25) and (225:0.5) .. (225:1);
\end{tikzpicture}
\qquad
\begin{tikzpicture}[vector]
\fill[fill=blue!30] (0,0) circle (1cm);
\draw[decoration={markings,mark=at position 0.5 with {\arrow[ultra thick]{>}}},postaction={decorate}] (135:1) -- (315:0.25);
\draw (315:0.25) to [out=-45,in=270] (0.75,0);
\fill[fill=blue!30] (0,0) circle (0.2cm);
\draw (225:1) -- (45:0.25);
\draw (45:0.25) to [out=45,in=90] (0.75,0);
\end{tikzpicture}
\raisebox{0.9cm}{\,$=z^{-1}$\,}
\begin{tikzpicture}[vector]
\fill[fill=blue!30] (0,0) circle (1cm);
\draw[decoration={markings,mark=at position 1 with {\arrow[ultra thick]{>}}},postaction={decorate}] (135:1) .. controls (135:0.5) and (-0.25,0.25) .. (-0.25,0);
\draw (-0.25,0) .. controls (-0.25,-0.25) and (225:0.5) .. (225:1);
\end{tikzpicture}
\end{equation*}
\caption{Basic relations}\label{fig:hb}
\end{figure}

The second example we take the Brauer category, \cite{MR1503378}, which is also the cobordism category of unoriented one-manifolds.  This is the free rigid symmetric monoidal on a symmetric self-dual object.

Then $\wedge^2 [1]$ is a metric Lie algebra.

The quantisation is the quotient of the category of oriented tangles by the Kauffman skein relation
given in Figure~\ref{fig:kf}.

The value of a simple loop is $\delta+1$.
\begin{equation}
\begin{tikzpicture}[red,line width=2pt]
\fill[fill=blue!30] (0,0) circle (1cm);
\fill[fill=blue!30] (3.5,0) circle (1cm);
\draw (0,0) circle(0.5);
\path (1.75,0) node[black] {$\displaystyle =\delta+1$};
\end{tikzpicture}\end{equation}

\begin{figure}\begin{equation*}
\begin{tikzpicture}[red,line width=2pt]
\fill[blue!30] (0,0) circle (1cm);
\draw (45:1) -- (225:1);
\fill[blue!30] (0,0) circle (0.25cm);
\draw (135:1) -- (315:1);
\end{tikzpicture}
\raisebox{0.9cm}{\,$-$\,}
\begin{tikzpicture}[red,line width=2pt]
\fill[blue!30] (0,0) circle (1cm);
\draw (135:1) -- (315:1);
\fill[blue!30] (0,0) circle (0.25cm);
\draw (45:1) -- (225:1);
\end{tikzpicture}
\raisebox{0.9cm}{\,$=(q-q^{-1})\Bigg($\,}
\begin{tikzpicture}[vector]
\fill[blue!30] (0,0) circle (1cm);
\draw (45:1) .. controls (45:0.5) and (315:0.5) .. (315:1);
\draw (135:1) .. controls (135:0.5) and (225:0.5) .. (225:1);
\end{tikzpicture}
\raisebox{0.9cm}{\,$-$\,}
\begin{tikzpicture}[red,line width=2pt]
\fill[blue!30] (0,0) circle (1cm);
\draw (45:1) .. controls (45:0.5) and (135:0.5) .. (135:1);
\draw (315:1) .. controls (315:0.5) and (225:0.5) .. (225:1);
\end{tikzpicture}
\raisebox{0.9cm}{\,$\Bigg)$}
\end{equation*}
\caption{Skein relation}\label{fig:kf}
\end{figure}

\begin{equation}
\begin{tikzpicture}[vector]
\fill[fill=blue!30] (0,0) circle (1cm);
\draw (225:1) -- (45:0.25);
\draw (45:0.25) .. controls (45:0.5) and (0.75,0.2) .. (0.75,0);
\fill[fill=blue!30] (0,0) circle (0.25cm);
\draw (135:1) -- (315:0.25);
\draw (315:0.25) .. controls (315:0.5) and (0.75,-0.2) .. (0.75,0);
\end{tikzpicture}
\raisebox{0.9cm}{\,$=z$\,}
\begin{tikzpicture}[vector]
\fill[fill=blue!30] (0,0) circle (1cm);
\draw (135:1) .. controls (135:0.5) and (-0.25,0.25) .. (-0.25,0);
\draw (-0.25,0) .. controls (-0.25,-0.25) and (225:0.5) .. (225:1);
\end{tikzpicture}
\qquad
\begin{tikzpicture}[vector]
\fill[fill=blue!30] (0,0) circle (1cm);
\draw (135:1) -- (315:0.25);
\draw (315:0.25) .. controls (315:0.5) and (0.75,-0.2) .. (0.75,0);
\fill[fill=blue!30] (0,0) circle (0.25cm);
\draw (225:1) -- (45:0.25);
\draw (45:0.25) .. controls (45:0.5) and (0.75,0.2) .. (0.75,0);
\end{tikzpicture}
\raisebox{0.9cm}{\,$=z^{-1}$\,}
\begin{tikzpicture}[vector]
\fill[fill=blue!30] (0,0) circle (1cm);
\draw (135:1) .. controls (135:0.5) and (-0.25,0.25) .. (-0.25,0);
\draw (-0.25,0) .. controls (-0.25,-0.25) and (225:0.5) .. (225:1);
\end{tikzpicture}
\end{equation}

\clearpage
\section{Diagrams and categories}\label{sec:vogel}
In this section we develop the PROP approach to metric Lie algebras, \cite{MR0171826}.
The first place in which the diagrammatic notation is used for the definition of a metric Lie algebra
is (4.46) and (4.48) in \cite[\S 4.5]{MR2418111}. This is then developed in \cite{MR1318886}.

\subsection{Graphs}
\begin{defn}
	A \emph{pre-graph} $\Gamma$ consists of two finite sets $B$ and $F$ a fixed point free involution $e$ on
	$F$ and a map $\sigma\colon B\sqcup F\rightarrow B\sqcup F$ such that each orbit of $\sigma$
	has order two or three.
A \emph{graph} $\Gamma$ is a pre-graph such that each orbit $\{x,y\}$ of $\sigma$
	of order two satisfies either $\{x,y\}\subset B$ or $\{x,y\}$ is an orbit of $e$.
\end{defn}
Any pre-graph can be reduced to a canonical graph by removing orbits of size two which violate the condition for a graph.

\begin{defn} For any finite set $X$, the set $\cG(X)$ is the set of
isomorphism classes of graphs with $B=X$.
\end{defn}
This defines a functor from the category of finite sets and bijections
to the category of sets.
The skeleton of the category of finite sets and bijections has objects
$[n]$ for $n\in\bN$ and morphisms permutations. This category is denoted by $\cS$
and is an initial object in the category of PROPs. The restriction of the functor to $\cS$ gives the sequence of sets $\cG([n])$ together with an action of $\fS_n$ on $\cG([n])$ for each $n$.

Let $\Gamma$ be a pre-graph and assume we are given a finite set $X$ and
two disjoint inclusions $X\rightarrow B$, which we write as $x\mapsto x_+$
with image $X_+$ and $x\mapsto x_-$ with image $X_-$.
Then we construct a new pre-graph $\Gamma^\prime$ by glueing.
The set $F^\prime$ is $F\cup X_+\cup X_-$ and the set $B^\prime$ is
$B/(X_+\cup X_-)$. Put $\sigma^\prime = \sigma$. Then we define $e^\prime$ 
to be $e$ on $F$ and on $X_+\cup X_-$ it is defined by $e^\prime\colon x_\pm\mapsto x_\mp$.

\begin{defn} The category $\widetilde{\cG}$ has objects finite sets. Let $X$ and $Y$ be finite
	sets. Then a morphism $X\rightarrow Y$ is a graph $\Gamma$ together with a
	bijection $X\sqcup Y\rightarrow B$. Let $\Gamma_1\colon X\rightarrow Y$
	and $\Gamma_1\colon Y\rightarrow Z$ be morphisms. Let $\Gamma$ be the disjoint
	union of the two graphs. Then we have two disjoint inclusions of $Y$ in $B$
	defined by the composites $Y\rightarrow B_{1,2}\rightarrow B$. Then the composite
	$\Gamma_1\circ \Gamma_2$ is obtained by first glueing and then reducing this
	pre-graph to a graph.
\end{defn}

Let $\cG$ be the full subcategory on the objects $\{[n]|n\in\bN\}$.
Then the inclusion $\cG\rightarrow\widetilde{\cG}$ is fully faithful
and essentially surjective, so is an equivalence.

First we give a basic example.
Let $\pi\colon X\rightarrow Y$ be a bijection. Then we construct a morphism $\gamma(\pi)$.
The sets $B(\pi)$ and $F(\pi)$ are both $X\sqcup Y$. The inclusions $X,Y$ in $B$ and $F$ are
denoted by $x,y\mapsto x_B,y_B$ and $x,y\mapsto x_F,y_F$. The map $\sigma$ is the involution
defined by $\sigma(x_F)=\pi(x)_F$ and $\sigma(y_F)=\pi^{-1}(y)_F$. Then we
define the involution $e$ by $e\colon x_B\leftrightarrow x_F$ and $e\colon y_B\leftrightarrow y_F$.

Then $\gamma$ can be interpreted as an inclusion of PROPs $\cS\rightarrow\cG$.

Every object is self-dual since for any identity map $\pi\colon X\rightarrow X$
we can regard $\gamma(\pi)$ either as a morphism $\emptyset\rightarrow X\sqcup X$
or as a morphism $X\sqcup X\rightarrow \emptyset$. Any morphism $X\rightarrow Y$
can also be regarded as a morphism $Y\rightarrow X$ and this defines an anti-involution
on $\cG$. 

This shows that $\widetilde{\cG}$ and $\cG$ are equivalent rigid
symmetric monoidal categories. In particular $\cG$ is a PROP.
An algebra for this PROP is a \emph{symmetric algebra} which is a vector space with a bilinear multiplication and a non-degenerate invariant symmetric form. This point of view gives a finite presentation of $\cG$.

Let $\Gamma$ be a pre-graph and take the equivalence relation on $B\cup F$ 
generated by $x\equiv \sigma(x)$ and $x\equiv e(x)$. Then the equivalence classes
are the connected components of $\Gamma$ and $\Gamma$ is the disjoint union of
its connected components.

Note that any connected graph which has a vertex does not have any orbit of
$\sigma$ of order two. There are only three connected graphs which do not have a vertex:
\begin{itemize}
	\item The empty graph.
	\item The graph $\gamma(\pi)$ where $\pi$ is the identity map of a
	set with one element.
	\item The graph with $B=\emptyset$, $F=\{x,y\}$ and both
	$\sigma$ and $e$ are the involution $x\leftrightarrow y$.
\end{itemize}

\subsection{Diagrams}
In this section we relate the construction in the previous section with the
conventional description of $\cG$.

\begin{defn}
	A pre-graph $\Gamma$ has a geometric realisation $|\Gamma|$. This is defined by
	\[ |\Gamma| = (B\times [0,1])\sqcup (F\times [0,1]) /\sim \]
	where $\sim$ is the equivalence relation generated by $(x,0)\sim (e(x),0)$ and 
	$(x,1)\sim (\sigma(x),1)$.
\end{defn}
This graph has univalent vertices corresponding to $B$,
bivalent vertices corresponding to orbits of $\sigma$ of order two and trivalent vertices
corresponding to orbits of $\sigma$ of order three. A reduction simply erases a bivalent
vertex. 

The map $\sigma$ corresponds to a cyclic ordering of the three edges at a trivalent
vertex. 
Conversely if we are given a graph such that all vertices are univalent
or bivalent and the three edges at each trivalent vertex are cyclically ordered then
we can reconstruct the graph. It is sufficient to define this for connected 
graphs. There are three connected graphs with no trivalent vertex. These
are the empty graph; the interval; and a loop. These we have already discussed.
Let $G$ be a connected graph with a trivalent vertex (and no bivalent vertex). 
Then we take $B$ to be the set
of univalent vertices. The set $F$ consists of the pairs consisting of an edge which
does not contain a univalent vertex and an orientation of the edge. The map $e$
reverses the orientation.

For example, the three connected graphs with no trivalent vertex are:
\begin{center}
\begin{tikzpicture}[red,line width=2pt]
\fill[fill=blue!30] (0,0) circle (1cm);
\end{tikzpicture}\qquad
\begin{tikzpicture}[red,line width=2pt]
\fill[fill=blue!30] (0,0) circle (1cm);
\draw (0,1) -- (0,-1);
\end{tikzpicture}\qquad
\begin{tikzpicture}[red,line width=2pt]
\fill[fill=blue!30] (0,0) circle (1cm);
\draw (0,0) circle (0.5cm);
\end{tikzpicture}
\end{center}

\subsection{Lie algebras}
\begin{defn} A Lie algebra in a $\bZ$-linear symmetric monoidal category
	is an algebra $L$ such that the multiplication map is anti-commutative
	and satisfies the Jacobi identity.
\end{defn}

\begin{defn}
	A metric Lie algebra in a $\bZ$-linear symmetric monoidal category is a symmetric algebra and a Lie algebra.
\end{defn}

Next we construct the PROP for metric Lie algebras.
First we take the free $\bZ[\delta]$-linear symmetric monoidal category on $\cG$.
Then we construct the universal example by imposing anti-symmetry and the Jacobi identity on the multiplication.
The condition that the multiplication is anti-commutative is known as the AS relation and is given in~\eqref{AS}.

\begin{equation}\label{AS}\begin{tikzpicture}[vector]
\draw (0,-0.4) to [out=120,in=270] (-0.3,0) to [out=90,in=-135] (45:1);
\draw (0,-0.4) to [out=60,in=270] (0.3,0) to [out=90,in=-45] (135:1);
\draw[-] (0,-0.4) -- (270:1);
\path (1.5,0) node[black] {$\displaystyle =-$};
\draw[-,xshift=3cm] (0,0) -- (45:1);
\draw[-,xshift=3cm] (0,0) -- (135:1);
\draw[-,xshift=3cm] (0,0) -- (270:1);
\begin{pgfonlayer}{background}
\filldraw [fill=blue!30,draw=blue!50]
(0,0) circle (1cm);
\filldraw [fill=blue!30,draw=blue!50]
(3,0) circle (1cm);
\end{pgfonlayer}
\end{tikzpicture}\end{equation}

Finally we have the Jacobi identity, and is also known as the $IHX$ relation. This is shown in~\eqref{IHX}.

\begin{equation}\label{IHX}\begin{tikzpicture}[red,line width=2pt]
\draw[-]  (0.2,0) --  (45:1);
\draw[-] (-0.2,0) -- (135:1);
\draw[-] (-0.2,0) -- (225:1);
\draw[-]  (0.2,0) -- (315:1);
\draw[-] (-0.2,0) -- (0.2,0);
\path (1.5,0) node[black] {$\displaystyle =$};
\draw[-,xshift=3cm] (0,0.2) -- (45:1);
\draw[-,xshift=3cm] (0,0.2) -- (135:1);
\draw[-,xshift=3cm] (0,-0.2) -- (225:1);
\draw[-,xshift=3cm] (0,-0.2) -- (315:1);
\draw[-,xshift=3cm] (0,-0.2) -- (0,0.2);
\path (4.5,0) node[black] {$\displaystyle +$};
\draw[-,xshift=6cm] (-0.3,-0.3) -- (0.2,0.2) -- (45:1);
\draw[-,xshift=6cm] (0.2,-0.2) -- (-0.2,0.2) -- (135:1);
\draw[-,xshift=6cm] (-0.2,-0.2) -- (225:1);
\draw[-,xshift=6cm] (0.2,-0.2) -- (315:1);
\draw[-,xshift=6cm] (-0.3,-0.3) -- (0.3,-0.3);
\begin{pgfonlayer}{background}
\filldraw [fill=blue!30,draw=blue!50]
(0,0) circle (1cm);
\filldraw [fill=blue!30,draw=blue!50]
(3,0) circle (1cm);
\filldraw [fill=blue!30,draw=blue!50]
(6,0) circle (1cm);
\end{pgfonlayer}
\end{tikzpicture}\end{equation}

This is the definition of $\delta$.
\begin{equation}\label{defd}\begin{tikzpicture}[red,line width=2pt]
\draw[-] (0,0) circle(0.5);
\path (1.5,0) node[black] {$\displaystyle =\delta$};
\begin{pgfonlayer}{background}
\filldraw [fill=blue!30,draw=blue!50]
(0,0) circle (1cm);
\filldraw [fill=blue!30,draw=blue!50]
(3,0) circle (1cm);
\end{pgfonlayer}
\end{tikzpicture}\end{equation}

This is a consequence of the two relations that say that the inner product
is symmetric and the multiplication is anti-symmetric.
\begin{equation}\label{tad}
\begin{tikzpicture}[red,line width=2pt]
\draw[-] (0,0) circle(0.5);
\draw[-] (0,0.5) -- (0,1);
\path (1.5,0) node[black] {$\displaystyle = 0$};
\begin{pgfonlayer}{background}
\filldraw [fill=blue!30,draw=blue!50]
(0,0) circle (1cm);
\end{pgfonlayer}
\end{tikzpicture}\end{equation}

\begin{defn} For each finite set $X$, the module $\cD(X)$ is the
quotient of $\bZ[\delta].\cG(X)$ by the relations 
\eqref{AS},~\eqref{IHX},~\eqref{defd}.
\end{defn}
This defines a functor from the category of finite sets and bijections
to the category of modules.

The restriction of the functor to the skeleton gives the sequence of modules $\wD([n])$ together with an action of $\fS_n$ on $\wD([n])$ for each $n$.

The category $\wD$ is the quotient of $\bZ[\delta].\cG$ by the relations~\eqref{AS},~\eqref{IHX},~\eqref{defd}. This means that $\wD$ is a
$\bZ[\delta]$-linear rigid symmetric monoidal category.
In particular, $\wD$ is a PROP.
An algebra for this PROP is a metric Lie algebra. This point of view gives a finite presentation of $\wD$.

Note that $\cG$ is graded by the number of trivalent vertices in a graph. The relations~\eqref{AS},~\eqref{IHX},~\eqref{defd} are homogeneous and so this descends to a grading on
$\wD$.

\subsection{Representations}
The main result of this section is closely related to (4.36) in \cite[\S 4.4]{MR2418111}.

Let $L$ be a Lie algebra in a linear symmetric monoidal category.
Then an \emph{$L$-module}, $V$, is an object in the category together
with a morphism $\mu_V\colon L\otimes V\rightarrow V$ which satisfies the STU relation.
For example, $L$ is itself an $L$-module. Also $I$ is the trivial $L$-module.

The tensor product of two $L$-modules is an $L$-module and so the category of $L$-modules
is monoidal. Let $U$ and $V$ be $L$-modules then $U\otimes V$ is the $L$-module
where $\mu_{U\otimes V}$ is the sum of two terms. The first term is $\mu_U\otimes\mathrm{id}_V$
and the second is the composite
\begin{equation*}
L\otimes U\otimes V \rightarrow L\otimes V\otimes U \overset{\mu_V\otimes\mathrm{id}_U}{\rightarrow} V\otimes U \rightarrow U\otimes V
\end{equation*}

The following Lemma will be used in the construction of the universal Lie algebra.
This is \cite[Lemma 3.1]{MR1865711}.

For each $n$ we define $\rho(n)\colon [n+1]\rightarrow [n]$.
These are defined by $\rho(0)=0$ and for $n\ge 1$,
\begin{equation*}
	\raisebox{1cm}{$\displaystyle\rho(n)=\sum_{i+j=n-1}$}\;\;
	\begin{tikzpicture}
	\fill [fill=blue!30] rectangle (2,1);
	\draw[vector] (0.75,1) node[anchor=south] {$i$} to [out=270,in=90] (0.25,0) node[anchor=north] {$i$};
	\draw[vector] (1.75,1) node[anchor=south] {$j$} -- (1.75,0) node[anchor=north] {$j$};
	\draw[vector] (0.25,1) to [out=270,in=120] (1,0.25);
	\draw[vector] (1.5,1) to [out=270,in=60] (1,0.25);
	\draw[vector] (1,0.25) -- (1,0);
	\end{tikzpicture}
\end{equation*}

\begin{lemma}\label{lemma1}
	We show that for any $\psi\colon [n]\rightarrow [m]$ that
	the following diagram commutes
	\[ \begin{CD}
	[n+1] @>{\rho(n)}>> [n] \\
	@V{\psi\otimes 1}VV @VV{\psi}V \\
	[m+1] @>>{\rho(m)}> [m]
	\end{CD} \]
	
	Then $\alpha(n)$ defines a representation of $L$ on $\otimes^n L$.
	This says that any diagram $\psi$ is a homomorphism of $L$-modules.
\end{lemma}

\begin{proof}
	We use the description of $\wD$ as a finitely generated monoidal
	category.
		
	Consider the set of morphisms which satisfy this condition.
	It is trivial that the identity morphisms satisfy this condition
	and that if two composable morphisms satisfy this condition then the
	composite satisfies this condition. Also if two morphisms satisfy this
	condition then the tensor product does also. These observations show that
	the set of morphisms which satisfy this condition is a monoidal subcategory.
	Hence it is sufficient to check this condition for the generators.
	This is straightforward. For the trivalent vertex this is equivalent to the
	Jacobi identity.
\end{proof}
\clearpage
\subsection{The universal Lie algebra}
In this section we give the construction of the universal Lie algebra.
This is due to Vogel although this has not been published. 

We now assume that we are working over $\bQ$. In fact for the construction of the
universal Lie algebra we only need to assume that $6$ is invertible.

\subsection{The Vogel ring}

For each triple $(D,v,\pi)$ where $D\in \cG(X)$, $v$ is a vertex of $D$ and 
$\pi\colon [3]\rightarrow v$ is a bijection we construct a map
\[ \psi(D,v,\pi)\colon \cG([3])\rightarrow \cG(X) \]

Let $T\in \cG(X)$. We construct $\psi(D,v,\pi)(T)$ as follows.

First we remove $v$. Put $v_F=v\cap F$ and $v_B=v\cap B$. Then 
\[ B^\prime = (B\cup e(v_F))/( v_B)\qquad F^\prime = F/(v_F\cup e(v_F)) \]
and $\sigma^\prime$ and $e^\prime$ are the restrictions of $\sigma$ and $e$.

We have an inclusion of $v_F$ in $B^\prime$ and $[3]$ by restricting $\pi$.
The graph of $T$ is obtained from the disjoint union of $D^\prime$ and $T$
by glueing these points. The boundary points of this graph are the disjoint union
of $B/v_B$ and $[3]/\pi^{-1}(v_F)$. This can be identified with $X$.

This map is compatible with the anti-symmetry relation and the Jacobi relation
and so descends to a linear map $\psi(D,v,\pi)\colon \wD([3])\rightarrow \wD(X)$.

\begin{defn} Let $\varepsilon\colon \fS_3\rightarrow\bZ$
	be the sign representation and define $\Lambda$ by
	\[ \Lambda = \{ D\in \wD([3]) | \sigma D=\varepsilon(\sigma)D \} \]
	for all $\sigma\in \fS_3$.
\end{defn}

The restriction of the map $\psi(D,v,\pi)$ to $\Lambda$ is independent of $\pi$ and so we have a
map 
\[ \psi(D,v)\colon \Lambda\rightarrow \wD(X) \]

\begin{prop} For all $D\in\Lambda$ and whenever $u$ and $v$ are connected by and edge
	we have $\psi(D,v)=\psi(D,u)$.\end{prop}
\begin{proof} The Proposition is equivalent to
	\begin{equation}\label{eqv}\begin{tikzpicture}[red,line width=2pt]
	\draw[-]  (0.2,0) --  (45:1);
	\draw[-] (-0.2,0) -- (135:1);
	\draw[-] (-0.2,0) -- (225:1);
	\draw[-]  (0.2,0) -- (315:1);
	\draw[-] (-0.2,0) -- (0.2,0);
	\draw[-] (-0.2,0) circle (0.1);
	\path (1.5,0) node[black] {$\displaystyle =$};
	\draw[-,xshift=3cm] (0.2,0) -- (45:1);
	\draw[-,xshift=3cm] (-0.2,0) -- (135:1);
	\draw[-,xshift=3cm] (-0.2,0) -- (225:1);
	\draw[-,xshift=3cm] (0.2,0) -- (315:1);
	\draw[-,xshift=3cm] (-0.2,0) -- (0.2,0);
	\draw[-,xshift=3cm] (0.2,0) circle (0.1);
	\begin{pgfonlayer}{background}
	\fill[fill=blue!30]
	(0,0) circle (1cm);
	\fill[fill=blue!30]
	(3,0) circle (1cm);
	\end{pgfonlayer}
	\end{tikzpicture}\end{equation}
	
	Applying Lemma~\ref{lemma1} twice gives that:
	\begin{equation}\label{eqa}\begin{tikzpicture}[red,line width=2pt]
	\draw[-]  (0.2,0) --  (45:1);
	\draw[-] (-0.2,0) -- (135:1);
	\draw[-] (-0.2,0) -- (225:1);
	\draw[-]  (0.2,0) -- (315:1);
	\draw[-] (-0.2,0) -- (0.2,0);
	\draw[-] (-0.2,0) circle (0.1);
	\path (1.5,0) node[black] {$\displaystyle =$};
	\draw[-,xshift=3cm] (0,0.2) -- (45:1);
	\draw[-,xshift=3cm] (0,0.2) -- (135:1);
	\draw[-,xshift=3cm] (0,-0.2) -- (225:1);
	\draw[-,xshift=3cm] (0,-0.2) -- (315:1);
	\draw[-,xshift=3cm] (0,-0.2) -- (0,0.2);
	\draw[-,xshift=3cm] (0,0.2) circle (0.1);
	\path (4.5,0) node[black] {$\displaystyle -$};
	\draw[-,xshift=6cm] (225:1) -- (45:1);
	\draw[-,xshift=6cm] (135:1) -- (315:1);
	\draw[-,xshift=6cm] (225:0.5) circle (0.1);
	\draw[-,xshift=6cm] (135:0.5cm) arc(135:225:0.5cm);
	\begin{pgfonlayer}{background}
	\fill[fill=blue!30]
	(0,0) circle (1cm);
	\fill[fill=blue!30]
	(3,0) circle (1cm);
	\fill[fill=blue!30]
	(6,0) circle (1cm);
	\end{pgfonlayer}
	\end{tikzpicture}\end{equation}
	\begin{equation}\label{eqb}\begin{tikzpicture}[red,line width=2pt]
	\draw[-]  (0.2,0) --  (45:1);
	\draw[-] (-0.2,0) -- (135:1);
	\draw[-] (-0.2,0) -- (225:1);
	\draw[-]  (0.2,0) -- (315:1);
	\draw[-] (-0.2,0) -- (0.2,0);
	\draw[-] (0.2,0) circle (0.1);
	\path (1.5,0) node[black] {$\displaystyle =$};
	\draw[-,xshift=3cm] (0,0.2) -- (45:1);
	\draw[-,xshift=3cm] (0,0.2) -- (135:1);
	\draw[-,xshift=3cm] (0,-0.2) -- (225:1);
	\draw[-,xshift=3cm] (0,-0.2) -- (315:1);
	\draw[-,xshift=3cm] (0,-0.2) -- (0,0.2);
	\draw[-,xshift=3cm] (0,-0.2) circle (0.1);
	\path (4.5,0) node[black] {$\displaystyle -$};
	\draw[-,xshift=6cm] (225:1) -- (45:1);
	\draw[-,xshift=6cm] (315:1) -- (135:1);
	\draw[-,xshift=6cm] (45:0.5) circle (0.1);
	\draw[-,xshift=6cm] (-45:0.5cm) arc(-45:45:0.5cm);
	\begin{pgfonlayer}{background}
	\fill[fill=blue!30] (0,0) circle (1cm);
	\fill[fill=blue!30] (3,0) circle (1cm);
	\fill[fill=blue!30] (6,0) circle (1cm);
	\end{pgfonlayer}
	\end{tikzpicture}\end{equation}
	
	Subtracting~\eqref{eqb} from~\eqref{eqa} and noting that the last diagrams in the
	two equations are equal gives:
	\begin{equation}\label{eqc}\begin{tikzpicture}[red,line width=2pt]
	\draw[-]  (0.2,0) --  (45:1);
	\draw[-] (-0.2,0) -- (135:1);
	\draw[-] (-0.2,0) -- (225:1);
	\draw[-]  (0.2,0) -- (315:1);
	\draw[-] (-0.2,0) -- (0.2,0);
	\draw[-] (-0.2,0) circle (0.1);
	\path (1.5,0) node[black] {$\displaystyle -$};
	\draw[-,xshift=3cm] (0.2,0) -- (45:1);
	\draw[-,xshift=3cm] (-0.2,0) -- (135:1);
	\draw[-,xshift=3cm] (-0.2,0) -- (225:1);
	\draw[-,xshift=3cm] (0.2,0) -- (315:1);
	\draw[-,xshift=3cm] (-0.2,0) -- (0.2,0);
	\draw[-,xshift=3cm] (0.2,0) circle (0.1);
	\path (4.5,0) node[black] {$\displaystyle =$};
	\draw[-,xshift=6cm] (0,0.2) -- (45:1);
	\draw[-,xshift=6cm] (0,0.2) -- (135:1);
	\draw[-,xshift=6cm] (0,-0.2) -- (225:1);
	\draw[-,xshift=6cm] (0,-0.2) -- (315:1);
	\draw[-,xshift=6cm] (0,-0.2) -- (0,0.2);
	\draw[-,xshift=6cm] (0,0.2) circle (0.1);
	\path (7.5,0) node[black] {$\displaystyle -$};
	\draw[-,xshift=9cm] (0,0.2) -- (45:1);
	\draw[-,xshift=9cm] (0,0.2) -- (135:1);
	\draw[-,xshift=9cm] (0,-0.2) -- (225:1);
	\draw[-,xshift=9cm] (0,-0.2) -- (315:1);
	\draw[-,xshift=9cm] (0,-0.2) -- (0,0.2);
	\draw[-,xshift=9cm] (0,-0.2) circle (0.1);
	\begin{pgfonlayer}{background}
	\fill[fill=blue!30] (0,0) circle (1cm);
	\fill[fill=blue!30] (3,0) circle (1cm);
	\fill[fill=blue!30] (6,0) circle (1cm);
	\fill[fill=blue!30] (9,0) circle (1cm);
	\end{pgfonlayer}
	\end{tikzpicture}\end{equation}
	Rotating this gives:
	\begin{equation}\label{eqd}\begin{tikzpicture}[red,line width=2pt]
	\draw[-]  (0,0.2) --  (45:1);
	\draw[-] (0,0.2) -- (135:1);
	\draw[-] (0,-0.2) -- (225:1);
	\draw[-]  (0,-0.2) -- (315:1);
	\draw[-] (0,-0.2) -- (0,0.2);
	\draw[-] (0,-0.2) circle (0.1);
	\path (1.5,0) node[black] {$\displaystyle -$};
	\draw[-,xshift=3cm] (0,0.2) -- (45:1);
	\draw[-,xshift=3cm] (0,0.2) -- (135:1);
	\draw[-,xshift=3cm] (0,-0.2) -- (225:1);
	\draw[-,xshift=3cm] (0,-0.2) -- (315:1);
	\draw[-,xshift=3cm] (0,-0.2) -- (0,0.2);
	\draw[-,xshift=3cm] (0,0.2) circle (0.1);
	\path (4.5,0) node[black] {$\displaystyle =$};
	\draw[-,xshift=6cm] (0.2,0) -- (45:1);
	\draw[-,xshift=6cm] (-0.2,0) -- (135:1);
	\draw[-,xshift=6cm] (-0.2,0) -- (225:1);
	\draw[-,xshift=6cm] (0.2,0) -- (315:1);
	\draw[-,xshift=6cm] (-0.2,0) -- (0.2,0);
	\draw[-,xshift=6cm] (-0.2,0) circle (0.1);
	\path (7.5,0) node[black] {$\displaystyle -$};
	\draw[-,xshift=9cm] (0.2,0) -- (45:1);
	\draw[-,xshift=9cm] (-0.2,0) -- (135:1);
	\draw[-,xshift=9cm] (-0.2,0) -- (225:1);
	\draw[-,xshift=9cm] (0.2,0) -- (315:1);
	\draw[-,xshift=9cm] (-0.2,0) -- (0.2,0);
	\draw[-,xshift=9cm] (0.2,0) circle (0.1);
	\begin{pgfonlayer}{background}
	\fill[fill=blue!30] (0,0) circle (1cm);
	\fill[fill=blue!30] (3,0) circle (1cm);
	\fill[fill=blue!30] (6,0) circle (1cm);
	\fill[fill=blue!30] (9,0) circle (1cm);
	\end{pgfonlayer}
	\end{tikzpicture}\end{equation}
	Then comparing~\eqref{eqc} and~\eqref{eqd} gives twice the relation~\eqref{eqv}. 
	Since 2 is invertible by assumption this gives the relation~\eqref{eqv}.\end{proof}

This shows that $\psi(D,v)=\psi(D,u)$ if the vertices $u$ and $v$ are in the
same connected component of $D$.

\begin{defn} For each finite set $X$, $\cG^{(s)}(X)$ is the subset of $\cG(X)$ of connected graphs which have a trivalent vertex; and $\cD^{(s)}(X)$ is the submodule spanned by the image of the map $\cG^{(s)}(X)\rightarrow\wD(X)$.
\end{defn}
Note that $\cG^{(s)}([3])=\cG([3])$.

The map $\cG^{(s)}(X)\rightarrow\wD(X)$ is compatible with the antisymmetry
relation and the Jacobi relation and so defines a bilinear map
\[ \Lambda \times \cD^{(s)}(X) \rightarrow \cD^{(s)}(X) \]
Taking $X=[3]$ and restricting to $\Lambda\otimes\Lambda$ gives a linear map
\[ \Lambda\otimes\Lambda \rightarrow \Lambda \]
The associativity condition is satisfied and so $\Lambda$ is a ring and for each
$X$, $\cD^{(s)}(X)$ is a $\Lambda$-module. There is a unique element of $\cG([3])$
with one vertex and this is the unit of $\Lambda$.

Let the degree of a diagram be $(k-1)/2$ where $k$ is the number of trivalent vertices.
Then $\Lambda$ is a graded ring and $\cD^{(s)}(X)$ is a graded $\Lambda$-module. 

A simple observation is that if $D$ is a diagram with at least two trivalent vertices
then $\alpha\beta D=\beta\alpha D$ for all $\alpha\beta\in\Lambda$.
Vogel has also shown that $12ab=12ba$ for any $a,b\in\Lambda_\bZ$.
Hence $\Lambda$ is a commutative ring since we have assumed that $12$ is invertible.

Let $t$, $s$, $p$ be independent indeterminates
of degrees 1,2,3 respectively. Set $\omega=p+s\,t+2\,t^3$.
Then define the graded commutative ring $K$ by
\[ K= \bQ[t]\oplus \omega\bQ[t,s,p] \subset \bQ[t,s,p] \]
The elements of $K$ up to degree 7 are:
\begin{center}
	\begin{tabular}{cccccccc}
		0 & 1 & 2 & 3 & 4 & 5 & 6 & 7 \\ \hline
		$1$ & $t$ & $t^2$ & $t^3$ & $t^4$ & $t^5$ & $t^6$ & $t^7$ \\
		& & & $\omega$ & $\omega t$ & $\omega t^2$ & $\omega t^3$ & $\omega t^4$ \\
		& & & & & $\omega s$ & $\omega st$ & $\omega st^2$ \\
		& & & & & & $\omega p$ & $\omega s^2$ \\
		& & & & & &  & $\omega pt$ \\
	\end{tabular}
\end{center}

The main result of \cite{MR1839694} is to construct a homomorphism of graded rings $K\rightarrow\Lambda$. This is known
to be an isomorphism for degree at most $9$ and in degree $10$ it is surjective with a one dimensional kernel, \cite{kneissler}.

For example, $t$ and $\omega$ are defined by
\begin{equation}\begin{tikzpicture}[red,line width=2pt]
\fill[fill=blue!30] (0,0) circle (1cm);
\draw[-] (0:0.5) -- (0:1);
\draw[-] (120:0.5) -- (120:1);
\draw[-] (240:0.5) -- (240:1);
\draw[-] (0:0.5) -- (120:0.5);
\draw[-] (120:0.5) -- (240:0.5);
\draw[-] (240:0.5) -- (0:0.5);
\path (1.5,0) node[black] {$\displaystyle =t$};
\end{tikzpicture}\qquad
\begin{tikzpicture}[red,line width=2pt]
\fill[fill=blue!30] (0,0) circle (1cm);
\draw (0,0) -- (0:0.5);
\draw (0,0) -- (120:0.5);
\draw (0,0) -- (240:0.5);
\draw (60:0.5) -- (60:1);
\draw (180:0.5) -- (180:1);
\draw (300:0.5) -- (300:1);
\draw (0:0.5) -- (60:0.5) -- (120:0.5) -- (180:0.5) -- (240:0.5) -- (300:0.5) -- cycle;
\end{tikzpicture}
\raisebox{0.85cm}{\quad $=4t^3-\frac32 \omega$}
\end{equation}

\subsection{The universal simple Lie algebra}
Let $X$ be a finite set. Then our aim is to construct a $\Lambda$-module $\cD(X)$
which is a quotient of the vector space $\wD(X)$ and such that the composite
map of vector spaces $\cD^{(s)}(X)\rightarrow\wD(X)\rightarrow\cD(X)$ is an inclusion
of $\Lambda$-modules.

\begin{defn} For each finite set $X$, $\cG^{(c)}(X)$ is the subset of $\cG(X)$ of connected graphs; and $\cD^{(c)}(X)$ is the submodule spanned by the image of the map $\cG^{(c)}(X)\rightarrow\wD(X)$.
\end{defn}

A basic observation is that for any
finite set $X$ there is a natural isomorphism
\begin{equation}
\wD(X) \cong \wD(\emptyset) \otimes \left( 
\oplus_\pi \otimes_{Y\in\pi} \cD^{(c)}(Y) \right) 
\end{equation} 
where the sum is over all partitions of $X$ into non-empty subsets and the tensor
products are over $\bQ$.

The vector space $\cD(\emptyset)$ is a $\bQ$-algebra under the disjoint union
and is the symmetric algebra on the graded vector space $\cD^{(c)}(\emptyset)$.
There is also an isomorphism
\begin{equation*}
\cD^{(c)}(\emptyset)\cong \bQ \oplus \cD^{(s)}(\emptyset)
\end{equation*}
A basis element for the factor $\bQ$ is the circle $\delta$ and $\cD^{(s)}(\emptyset)$
is the free $\Lambda$-module on $\Theta$, where
\begin{equation}
\raisebox{0.9cm}{$\Theta=$\;}
\begin{tikzpicture}
\fill[fill=blue!30] (0,0) circle (1cm);
\draw[vector] (0,0.5) -- (0,-0.5);
\draw[vector] (0,0.5) to [out=60,in=90] (0.5,0);
\draw[vector] (0,0.5) to [out=120,in=90] (-0.5,0);
\draw[vector] (0,-0.5) to [out=300,in=270] (0.5,0);
\draw[vector] (0,-0.5) to [out=240,in=270] (-0.5,0);
\end{tikzpicture}
\end{equation}

This constructs a surjective homomorphism
of $\bQ$-algebras $\wD(\emptyset)\rightarrow\Lambda[\delta]$.

\begin{defn}\label{defn:f}
For each finite set $X$, $\cD(X)$ is defined by by
\begin{equation}
\cD(X) \cong  \Lambda[\delta] \otimes \left( 
\oplus_\pi \otimes_{Y\in\pi} \cD^{(c)}(Y) \right) 
\end{equation} 
where the sum is over all partitions of $X$ into non-empty subsets and the tensor products
are over $\Lambda$.
\end{defn}

\begin{defn}
	The universal Lie algebra is the PROP with
	\begin{equation*}
		\Hom([n],[m])=\cD([n]\coprod [m])
	\end{equation*}
	and composition given by glueing.
\end{defn}

The following is a relation in this category.
\begin{equation}\begin{tikzpicture}[red,line width=2pt]
\draw[-] (0,0) circle(0.5);
\draw[-] (0,0.5) -- (0,1);
\draw[-] (0,-0.5) -- (0,-1);
\path (1.5,0) node[black] {$\displaystyle = 2t$};
\draw[-] (3,-1) -- (3,1);
\begin{pgfonlayer}{background}
\fill[fill=blue!30] (0,0) circle (1cm);
\fill[fill=blue!30] (3,0) circle (1cm);
\end{pgfonlayer}
\end{tikzpicture}\end{equation}

An algebra for this PROP is a metric Lie algebra such that
the canonical homomorphism $\End(I)\rightarrow \End(L)$ is an isomorphism
and such that $\Hom([0],[3])$ is the free $\End(I)$-module on a trivalent
vertex. This then gives a character $\Lambda[\delta]\rightarrow\End(I)$
by evaluation.

In order to construct the universal simple Lie algebra we need to form a quotient. The aim is to force a certain idempotent $\pi\in\End(\wedge^2L)$
to be zero. This is motivated by the observation that there are no known examples for which this is non-zero.

Let $\pi$ be the antisymmetriser in the group algebra of $\fS_2$.
We identify $\pi\End([2])\pi$ with $\End(\bigwedge^2[1])$.
The element $\pi$ is the identity in this algebra and $K$ is also
an element in this algebra. Introduce $w'$ by
\begin{equation}
\begin{tikzpicture}[vector]
\fill[fill=blue!30] (0,0) rectangle (2,1);
\draw (0.5,0.25) -- (1.5,0.25) -- (1.5,0.75) -- (0.5,0.75) -- cycle;
\draw (1,0.75) -- (1,0.25);
\draw (0.25,0) to [in=90,out=45] (0.5,0.25);
\draw (1.75,0) to [in=90,out=-45] (1.5,0.25);
\draw (0.25,1) to [in=270,out=45] (0.5,0.75);
\draw (1.75,1) to [in=270,out=-45] (1.5,0.75);
\end{tikzpicture}
\end{equation}
Put $w=\pi w'\pi$.

These elements satisfy
\begin{equation}
K^2=2tK\qquad
wK = (4t^3-\frac32 \omega)K = Kw
\end{equation}
A calculation due to Kneisler gives
\begin{equation}
\omega w^2 = -\frac32 p\omega w + 
(4t^3-\frac32 \omega)(\frac12 \omega t^2-\frac34 s\omega)K
\end{equation}
Define $E$ by
\begin{equation*}
E = (2t^2-\frac34 \frac{\omega}{t})K-w
\end{equation*}
then this satisfies $E^2=\frac 32 pE$.

This shows that, if $p$ is invertible, then
$\frac23 \frac1p E$ is an idempotent. The aim is
to impose the relation $E=0$. However this requires that $\dim E=0$ which is a relation on the scalars.

First we invert $p\omega$, so $\Lambda$ is replaced by $\Lambda[\frac1{p\omega]}]$ and the homomorphism 
 $\bQ[t]\oplus \omega\bQ[t,s,p]\rightarrow\Lambda$ is replaced by the homomorphism $\bQ[t,s,p,p^{-1}]\rightarrow\Lambda[\frac1{p\omega]}]$.
Next we extend the homomorphism $\Lambda\rightarrow\Lambda[\frac1{p\omega]}]$ to a homomorphism $\Lambda[\delta]\rightarrow\Lambda[\frac1{p\omega]}]$ by
\[ \delta\mapsto \frac{p-2st-4t^3}p \]
Using this homomorphism we can specialise the ring of scalars to
$\Lambda[\frac1{p\omega]}]$ and in this specialisation we can impose the relation $E=0$.
This gives the universal simple Lie algebra.

This is called the universal simple Lie algebra because every simple Lie algebra or simple Lie superalgebra is an algebra for this PROP.
The simple Lie superalgebras are discussed in \cite{MR1941724}. A more exotic algebra for this PROP is discussed in \cite{MR3154807}.

The following relation will be made use of in section \ref{sec:plane}. This uses the notation for
diagrams in \eqref{eq:diagrams}.
\begin{prop}
	The universal simple Lie algebra satisfies the relation \eqref{eq:H3}
\end{prop}

\begin{proof} The relation $E=0$ can be written as
	\begin{equation*}
		w=\left( \frac12 t^2-\frac34 s\right)K-\frac32 p\pi
	\end{equation*}
	Then we also have from \cite{univ} that
	\begin{equation*}
		w=w'+\frac12 H^3-\frac32 tH^2+\frac12 t^2K
	\end{equation*}
	Eliminating $w$ from these equations gives
	\begin{equation}\label{eq:cubic}
	w'+\frac12 H^3 - \frac32 tH^2 = -\frac34 sK - \frac34 (1-X)
	\end{equation}
	Under rotation through a quarter turn we have $w'\leftrightarrow H^3$, $H^2\leftrightarrow H^2$, $H\leftrightarrow K$ and $1\leftrightarrow U$.
	So rotating \eqref{eq:cubic} gives a further equation. Eliminating $w'$ from these two equations gives the equation for $H^3$.
\end{proof}

\section{Vogel plane}\label{sec:plane}
In this section we construct the Vogel plane by imposing relations on the
PROP for a metric Lie algebra.

Then we impose the conditions that
\begin{equation*}
\bigwedge^2 \fg \cong \fg \oplus X \qquad
S^2 \fg \cong 1 \oplus Y(\alpha) \oplus Y(\beta) \oplus Y(\gamma)
\end{equation*}

The two string centraliser algebra has dimension six and we assume the following are a basis:
\[ 1, U, K, H, X, H^2 \]
These correspond to the diagrams
\begin{equation}\label{eq:diagrams}
	\begin{tabular}{cccccc}
\begin{tikzpicture}[vector]
\fill[fill=blue!30] (0,0) rectangle (1,1);
\draw (0.25,1) -- (0.25,0);
\draw (0.75,1) -- (0.75,0);
\end{tikzpicture} &
\begin{tikzpicture}[vector]
\fill[fill=blue!30] (0,0) rectangle (1,1);
\draw (0.75,0) arc (0:180:0.25);
\draw (0.75,1) arc (0:-180:0.25);
\end{tikzpicture} &
\begin{tikzpicture}[vector]
\fill[fill=blue!30] (0,0) rectangle (1,1);
\draw (0.25,1) -- (0.5,0.65);
\draw (0.75,1) -- (0.5,0.65);
\draw (0.5,0.65) -- (0.5,0.35);
\draw (0.25,0) -- (0.5,0.35);
\draw (0.75,0) -- (0.5,0.35);
\end{tikzpicture} &
\begin{tikzpicture}[vector]
\fill[fill=blue!30] (0,0) rectangle (1,1);
\draw (0.25,1) -- (0.35,0.5);
\draw (0.75,1) -- (0.65,0.5);
\draw (0.65,0.5) -- (0.35,0.5);
\draw (0.25,0) -- (0.35,0.5);
\draw (0.75,0) -- (0.65,0.5);
\end{tikzpicture} &
\begin{tikzpicture}[vector]
\filldraw [fill=blue!30,draw=blue!50] (0,0) rectangle (1,1);
\draw (0.25,1) .. controls (0.25,0.5) and (0.75,0.5) .. (0.75,0);
\draw (0.75,1) .. controls (0.75,0.5) and (0.25,0.5) .. (0.25,0);
\end{tikzpicture} &
\begin{tikzpicture}[vector]
\fill[fill=blue!30] (0,0) rectangle (1,1);
\draw (0.35,0.35) rectangle (0.65,0.65);
\draw (0.25,1) -- (0.35,0.65);
\draw (0.75,1) -- (0.65,0.65);
\draw (0.25,0) -- (0.35,0.35);
\draw (0.75,0) -- (0.65,0.35);
\end{tikzpicture} \\
$1$ & $U$ & $K$ & $H$ & $X$ & $H^2$
	\end{tabular}
\end{equation}

Here $1$ is a unit and the remaining multiplication table is
\begin{center}
	\begin{tabular}{c|ccccc}
		& $U$ & $K$ & $H$ & $X$ & $H^2$ \\ \hline
		$U$ & $\delta U$ & 0 & $2t\, U$ & $U$ & $4t^2\, U$ \\
		$K$ & 0 & $2t K$ & $t\, K$ & $-K$ & $t^2\, K$ \\
		$H$ & $2t\, U$ & $t\, K$ & $H^2$ & $H-K$ & $H^3$ \\
		$X$ & $U$ & $-K$ & $H-K$ & 1 & $H^2-t\, K$ \\
		$H^2$ & $4t^2\, U$ & $t^2\, K$ & $H^3$ & $H^2-t\, K$ & $H^4$ \\
	\end{tabular}
\end{center}

Then we have the relations
\begin{equation}\label{digon}
\begin{tikzpicture}[red,line width=2pt]
\draw[-] (0,0) circle(0.5);
\draw[-] (0,0.5) -- (0,1);
\draw[-] (0,-0.5) -- (0,-1);
\path (1.5,0) node[black] {$\displaystyle = 2t$};
\draw[-,xshift=3cm] (90:1) -- (270:1);
\begin{pgfonlayer}{background}
\fill[fill=blue!30] (0,0) circle (1cm);
\fill[fill=blue!30] (3,0) circle (1cm);
\end{pgfonlayer}
\end{tikzpicture}\end{equation}

\begin{equation}\label{triV}
\begin{tikzpicture}[red,line width=2pt]
\draw (30:0.5) -- (30:1);
\draw (150:0.5) -- (150:1);
\draw (270:0.5) -- (270:1);
\draw (30:0.5) -- (150:0.5) -- (270:0.5) -- cycle;
\path (1.5,0) node[black] {$\displaystyle = t$};
\draw[-,xshift=3cm] (0,0) -- (30:1);
\draw[-,xshift=3cm] (0,0) -- (150:1);
\draw[-,xshift=3cm] (0,0) -- (270:1);
\begin{pgfonlayer}{background}
\fill[fill=blue!30] (0,0) circle (1cm);
\fill[fill=blue!30] (3,0) circle (1cm);
\end{pgfonlayer}
\end{tikzpicture}\end{equation}

\begin{equation}\label{triX}
\begin{tikzpicture}[red,line width=2pt]
\draw (30:0.5) -- (30:1);
\draw (150:0.5) -- (150:1);
\draw (270:0.5) -- (270:1);
\draw (30:0.5) -- (150:0.5) -- (270:0.5) -- cycle;
\path (1.5,0) node[black] {$\displaystyle = 0$};
\path (0,-1) node[black] {$X$};
\begin{pgfonlayer}{background}
\fill[fill=blue!30] (0,0) circle (1);
\end{pgfonlayer}
\end{tikzpicture}\end{equation}

The first relation fixes a normalisation and is equivalent to
saying that the Casimir acts by $2t$ on the adjoint representation.
The other two relations are consequences of the Jacobi identity.

Finally we have the relations
\begin{equation}\label{eq:H3}
H^3=t\,H^2 -s\,H + \frac s2\, K + \frac p2 \left( U+X-2 \right)
\end{equation}

The dimensions are given by
\begin{align}
\delta &= \frac{-4t^3+2st+p}{p} \\
\dim X &= -\delta \frac{2t^3+st+p}{p}
\end{align}

The Vogel plane is defined, as a PROP, by a finite presentation and the motivation for this presentation came from the universal simple Lie algebra. Therefore, by construction, we have a map of PROPs from the Vogel plane to the universal simple Lie algebra.

It is convenient to introduce a Galois extension of the ring of scalars. The ring of scalars is extended to $\bQ[\alpha,\beta,\gamma,\frac1{\alpha\beta\gamma}]$. The homomorphism from  $\bQ[s,t,p,p^{-1}]$ is given by
\[ t\mapsto\alpha+\beta+\gamma,
s\mapsto\alpha\,\beta+\beta\,\gamma+\gamma\,\alpha,
p\mapsto \alpha\,\beta\,\gamma \]

This gives the representations $Y(\alpha)$,$Y(\beta)$,$Y(\gamma)$
and we have the relations
\begin{equation}
\Hom(Y(\alpha)\otimes L,Y(\beta))=0
\end{equation}
and similarly for $(\alpha,\beta)$ replaced by any pair of distinct elements of $\{\alpha,\beta,\gamma\}$. In particular, this implies the following relation
\begin{equation}
\begin{tikzpicture}[red,line width=2pt]
\fill[fill=blue!30] (0,0) circle (1);
\draw (30:0.5) -- (30:1);
\draw (150:0.5) -- (150:1);
\draw (270:0.5) -- (270:1);
\draw (30:0.5) -- (150:0.5) -- (270:0.5) -- cycle;
\path (1.5,0) node[black] {$\displaystyle = 0$};
\path (0,0.5) node[black] {$\alpha$};
\path (0,-1) node[black] {$\beta$};
\end{tikzpicture}\end{equation}

The dimensions are given by
\begin{align}
\delta &= -\frac{(2t-\alpha)(2t-\beta)(2t-\gamma)}{\alpha\beta\gamma} \label{dimV} \\
\dim X &= -\delta \frac{(t+\alpha)(t+\beta)(t+\gamma)}{\alpha\beta\gamma} \label{dimX} \\
\dim Y(\alpha) &= \frac{t(2t-\beta)(2t-\gamma)(t+\beta)(t+\gamma)(2t-3\alpha)}%
{\alpha^2\beta\gamma(\alpha-\beta)(\alpha-\gamma)} \label{dimY}
\end{align}

\begin{prop}\label{prop:cas}
	The values of the Casimir are given by
	\begin{center}
		\begin{tabular}{cccccc}
			$\fg$ & $V$ & $X$ & $Y_\alpha$ & $Y_\beta$ & $Y_\gamma$ \\ \hline
			$t$ & $t$ & $2t$ & $(2t-\alpha)$ & $(2t-\beta)$ & $(2t-\gamma)$ \\			
		\end{tabular}
	\end{center}	
\end{prop}

The Casimir is $2t-H$. It would be interesting to also define the infinitesimal braiding.

\clearpage
\section{Quantisation of Vogel plane}\label{sec:planeq}
In this section we give the quantisation of the Vogel plane.
This was presented in \cite{MR2029689}. This is a finitely
presented ribbon category for which the presentation is a $q$-analogue of the presentation of the Vogel plane.

We introduce indeterminates $q^\alpha,q^\beta,q^\gamma$ and for integers $n_\alpha,n_\beta,n_\gamma$ we put
\begin{equation}\label{eq:qexp}
q^{n_\alpha\alpha+n_\beta\beta+n_\gamma\gamma}=
(q^{\alpha})^{n_\alpha}.(q^\beta)^{n_\beta}.(q^\gamma)^{n_\gamma}
\end{equation}
and we also use the notation
\begin{equation}\label{eq:qint}
[n_\alpha\alpha+n_\beta\beta+n_\gamma\gamma]=
q^{n_\alpha\alpha+n_\beta\beta+n_\gamma\gamma}-
q^{-n_\alpha\alpha-n_\beta\beta-n_\gamma\gamma}
\end{equation}
We will refer to these as quantum integers even though the standard definition has denominator $q-q^{-1}$.

The rational functions will all be linear combinations of rational functions in which the numerator and denominator are a product of quantum integers and the number of quantum integers in the numerator is at least the number of quantum integers in the denominator. In addition the rational functions are invariant under permutations of
the indeterminates $q^\alpha,q^\beta,q^\gamma$.

Rational functions of this form have a \emph{classical limit}. This can be defined analytically by regarding $q$ as a complex number and 
$q^\alpha,q^\beta,q^\gamma$ as powers of $q$ and then taking the limit $q\rightarrow 1$. Equivalently, this limit can be defined algebraically. Given a rational function in which the numerator and denominator are a product of quantum integers: if there are more quantum integers in the numerator than in the denominator then the classical limit is 0; otherwise, the classical limit is obtained by replacing each
$[n_\alpha\alpha+n_\beta\beta+n_\gamma\gamma]$ by
$n_\alpha\alpha+n_\beta\beta+n_\gamma\gamma$.

The most straightforward choice for the ring of scalars is the field of rational functions, $\bQ(q^\alpha,q^\beta,q^\gamma)$. The disadvantage of this choice of scalars is that it does not allow us to take specialisations. In particular, the classical limit is not given by a ring homomorphism. We conjecture that there is an integral form of the quantised Vogel plane, which is  a finitely presented ribbon category over a ring $K$. The intention is that $K$ has the following properties. First, $K$ is an integral domain whose field of fractions is 
$\bQ(q^\alpha,q^\beta,q^\gamma)$ and this specialisation gives the presentation in this section. Second the classical limit is given by a ring homomorphism and this specialisation gives the presentation of the Vogel plane. Each entry in Table~\ref{table:vogel} gives a homomorphism $K\rightarrow \bQ(q)$ by taking $q^\alpha,q^\beta,q^\gamma$ to be a power of the indeterminate $q$. Then the third property of $K$ is that each of these is, in fact, a homomorphism $K\rightarrow \bZ[q,q^{-1}]$.

The two string centraliser algebra still has dimension six but the
element $X$ is no longer available. In its place we have two elements,
the braid generator and its inverse.

The eigenvalues of the braiding $\sigma$ are given in Table~\ref{table:brq}.
\begin{figure}
	\begin{tabular}{cc|cccc}
		$\fg$ & $X$ & $1$ & $Y(\alpha)$ & $Y(\beta)$ & $Y(\gamma)$ \\ \hline
		$-q^{2t}$ & $-1$ & $q^{4t}$ & $q^{2\alpha}$ & $q^{2\beta}$ & $q^{2\gamma}$ 
	\end{tabular}
\caption{Eigenvalues of braid generator}\label{table:brq}
\end{figure}

The eigenvalues of $H$ are given by
\begin{center}
	\begin{tabular}{cccc}
	$1$ & $Y(\alpha)$ & $Y(\beta)$ & $Y(\gamma)$ \\ \hline
		 $1$ & $\frac{[\alpha][\beta+\gamma]}{[t][2\beta+2\gamma]}$ & $\frac{[\beta][\alpha+\gamma]}{[t][2\alpha+2\gamma]}$ & $\frac{[\gamma][\alpha+\beta]}{[t][2\alpha+2\beta]}$ 
	\end{tabular}
\end{center}
The eigenvalue associated to the representation $\fg$ is
\begin{equation}\label{eq:hk}
\frac{[\alpha][\beta+\gamma]}{[t][2\beta+2\gamma]} +\frac{[\beta][\alpha+\gamma]}{[t][2\alpha+2\gamma]} +\frac{[\gamma][\alpha+\beta]}{[t][2\alpha+2\beta]}
\end{equation}

The eigenvalue of $H$ associated to $X$ is
\begin{equation}\label{eq:hxq}
	- \frac{[\alpha][\beta][\gamma]}{[t]}
	\, \frac{[\beta+\gamma]}{[2\beta+2\gamma]}\, \frac{[\alpha+\gamma]}{[2\alpha+2\gamma]}\, \frac{[\alpha+\beta]}{[2\alpha+2\beta]}\,
\end{equation}

In terms of diagrams, we replace~\eqref{digon} by
\begin{equation}\label{digonq}
\begin{tikzpicture}[red,line width=2pt]
\draw[-] (0,0) circle(0.5);
\draw[-] (0,0.5) -- (0,1);
\draw[-] (0,-0.5) -- (0,-1);
\path (1.5,0) node[black] {$\displaystyle =$};
\draw[-,xshift=3cm] (90:1) -- (270:1);
\begin{pgfonlayer}{background}
\fill[fill=blue!30] (0,0) circle (1cm);
\fill[fill=blue!30] (3,0) circle (1cm);
\end{pgfonlayer}
\end{tikzpicture}\end{equation}
This is a choice of normalisation. In the classical limit this gives
a different normalisation to the one we used previously.

The constant $t$ in the relation~\eqref{triV} is replaced by~\eqref{eq:hk}.
The relations~\eqref{defd},~\eqref{tad},~\eqref{eq:orth} are unmodified.
The remaining relations which do not involve a crossing are~\eqref{eq:H3}
and~\eqref{triX}; these become significantly more complicated.

The antisymmetry relation~\eqref{AS} is replaced by the following pair of relations:
\begin{equation}\label{ASq}
\begin{tikzpicture}[red,line width=2pt]
\fill[fill=blue!30] (0,0) circle (1cm);
\draw (0,-0.4) to [out=120,in=270] (-0.3,0) to [out=90,in=-135] (45:1);
\fill[blue!30] (0,0.35) circle(0.2);
\draw (0,-0.4) to [out=60,in=270] (0.3,0) to [out=90,in=-45] (135:1);
\draw (0,-0.4) -- (270:1);
\end{tikzpicture}
\raisebox{0.9cm}{\,$\displaystyle =-q^{2t}$\,}
\begin{tikzpicture}[red,line width=2pt]
\fill[fill=blue!30] (0,0) circle (1cm);
\draw[-] (0,0) -- (45:1);
\draw[-] (0,0) -- (135:1);
\draw[-] (0,0) -- (270:1);
\end{tikzpicture}
\qquad
\begin{tikzpicture}[vector]
\fill[fill=blue!30] (0,0) circle (1cm);
\draw (0,-0.4) to [out=60,in=270] (0.3,0) to [out=90,in=-45] (135:1);
\fill[blue!30] (0,0.35) circle(0.2);
\draw (0,-0.4) to [out=120,in=270] (-0.3,0) to [out=90,in=-135] (45:1);
\draw (0,-0.4) -- (270:1);
\end{tikzpicture}
\raisebox{0.9cm}{\,$\displaystyle =-q^{-2t}$\,}
\begin{tikzpicture}[vector]
\fill[fill=blue!30] (0,0) circle (1cm);
\draw (0,0) -- (45:1);
\draw (0,0) -- (135:1);
\draw (0,0) -- (270:1);
\end{tikzpicture}
\end{equation}

The symmetry of the inner product is replaced by the following pair of relations:
\begin{equation}\label{eq:inn}
\begin{tikzpicture}[vector]
\fill[fill=blue!30] (0,0) circle (1cm);
\draw (-0.3,0) to [out=90,in=225] (45:1);
\fill[blue!30] (0,0.3) circle(0.2);
\draw (0.3,0) to [out=90,in=315] (135:1);
\draw (0.3,0) arc (360:180:0.3);
\end{tikzpicture}
\raisebox{0.9cm}{\,$\displaystyle =q^{4t}$\,}
\begin{tikzpicture}[vector]
\fill[fill=blue!30] (0,0) circle (1cm);
\draw (45:1) to [out=225,in=0] (0,0.1) to [out=180,in=315] (135:1);
\end{tikzpicture}
\qquad
\begin{tikzpicture}[vector]
\fill[fill=blue!30] (0,0) circle (1cm);
\draw (0.3,0) to [out=90,in=315] (135:1);
\fill[blue!30] (0,0.3) circle(0.2);
\draw (-0.3,0) to [out=90,in=225] (45:1);
\draw (0.3,0) arc (360:180:0.3);
\end{tikzpicture}
\raisebox{0.9cm}{\,$\displaystyle =q^{-4t}$\,}
\begin{tikzpicture}[red,line width=2pt]
\fill[fill=blue!30] (0,0) circle (1cm);
\draw (45:1) to [out=225,in=0] (0,0.1) to [out=180,in=315] (135:1);
\end{tikzpicture}
\end{equation}

The quantum dimensions are given by
\begin{equation}\label{dimVq}
\delta=-\frac{[2t-\alpha][2t-\beta][2t-\gamma]}{[\alpha][\beta][\gamma]}
\end{equation}
This gives~\eqref{dimV} in the classical limit.

\begin{multline}\label{dimXq}
\dim_q(X)=-\delta \frac{[t+\alpha][t+\beta][t+\gamma]}{[2\alpha][2\beta][2\gamma]}
\times \frac{[2\alpha+2\beta]}{[\alpha+\beta]}
\frac{[2\alpha+2\gamma]}{[\alpha+\gamma]}
\frac{[2\beta+2\gamma]}{[\beta+\gamma]}
\end{multline}
This gives~\eqref{dimX} in the classical limit.
\begin{equation}\label{dimYq}
\dim_q (Y(\alpha))=
\frac{[2t][2t-\alpha][2t-\beta][2t-\gamma][t+\beta][t+\gamma]}%
{[2\alpha][\alpha][\beta][\gamma][\alpha-\beta][\alpha-\gamma]}
\end{equation}
This gives~\eqref{dimY} in the classical limit.

Note that the dimension formulae in
\cite{MR2211533} can be interpreted as quantum dimensions.

The skein relation for the unoriented, framed braid group generators is given in Figure~\ref{skeinq}.
\begin{figure}\begin{multline*}
		\begin{tikzpicture}[vector]
		\fill[blue!30] (0,0) circle (1cm);
		\draw (45:1) -- (225:1);
		\fill[blue!30] (0,0) circle (0.25cm);
		\draw (135:1) -- (315:1);
		\end{tikzpicture}
		\raisebox{0.9cm}{\,$-$\,}
		\begin{tikzpicture}[vector]
		\fill[blue!30] (0,0) circle (1cm);
		\draw (135:1) -- (315:1);
		\fill[blue!30] (0,0) circle (0.25cm);
		\draw (45:1) -- (225:1);
		\end{tikzpicture}
		\raisebox{0.9cm}{\,$=$} \\
		\raisebox{0.9cm}{$[\alpha][\beta][\gamma]\frac{[t]}{[2t]}\,\Bigg($\,}
		\begin{tikzpicture}[red,line width=2pt]
		\fill[blue!30] (0,0) circle (1cm);
		\draw (45:1) .. controls (45:0.5) and (315:0.5) .. (315:1);
		\draw (135:1) .. controls (135:0.5) and (225:0.5) .. (225:1);
		\end{tikzpicture}
		\raisebox{0.9cm}{\,$-$\,}
		\begin{tikzpicture}[red,line width=2pt]
		\fill[blue!30] (0,0) circle (1cm);
		\draw (45:1) .. controls (45:0.5) and (135:0.5) .. (135:1);
		\draw (315:1) .. controls (315:0.5) and (225:0.5) .. (225:1);
		\end{tikzpicture}
		\raisebox{0.9cm}{\,$\Bigg)$} \\
		\raisebox{0.9cm}{$+\frac{[2\alpha+2\beta]}{[\alpha+\beta]}
			\frac{[2\alpha+2\gamma]}{[\alpha+\gamma]}
			\frac{[2\beta+2\gamma]}{[\beta+\gamma]}[t]$\,$\Bigg($\,}
		\begin{tikzpicture}
		\fill[fill=blue!30] (0,0) circle (1cm);
		\draw[vector]  (0.2,0) --  (45:1);
		\draw[vector] (-0.2,0) -- (135:1);
		\draw[vector] (-0.2,0) -- (225:1);
		\draw[vector]  (0.2,0) -- (315:1);
		\draw[adjoint] (-0.2,0) -- (0.2,0);
		\end{tikzpicture}
		\raisebox{0.9cm}{\,$-$\,}
		\begin{tikzpicture}
		\fill[fill=blue!30] (0,0) circle (1cm);
		\draw[vector] (0,0.2) -- (45:1);
		\draw[vector] (0,0.2) -- (135:1);
		\draw[vector] (0,-0.2) -- (225:1);
		\draw[vector] (0,-0.2) -- (315:1);
		\draw[adjoint] (0,-0.2) -- (0,0.2);
		\end{tikzpicture}
		\raisebox{0.9cm}{\,$\Bigg)$}
	\end{multline*}
	\caption{Skein relation}\label{skeinq}
\end{figure}

\subsection{Knot invariants}
The quantised Vogel plane is a ribbon category and so each object
gives rise to an invariant of framed links with values in $\End(I)$. This was the original motivation for introducing ribbon categories. This is described in \cite{MR1036112} and also follows from
the coherence theorem in \cite{MR1268782}. A natural question\footnote{which arose in a conversation with Vaselov} is whether the link invariant associated to $V$ can be computed in simple examples.
In this section we show that this link invariant can be computed
for the closures of two string braids.

The formula for the invariant of the closure of $\sigma^k$ is given in terms of the quantum dimensions and the eigenvalues in Table~\ref{table:brq} by the following expression:
\begin{multline}
(-1)^k\left(\delta\, q^{2kt}+\dim_q X\right) \\
+q^{4kt}+\dim_q(Y(\alpha))\, q^{2k\alpha}+\dim_q(Y(\beta))\, q^{2k\beta}
+\dim_q(Y(\gamma))\, q^{2k\gamma} \\
\end{multline}
This invariant is multiplicative under disjoint union of links.
It is more convenient to divide by $\delta$ to get an invariant
of non-empty links which is multiplicative under connected sum.

This is an invariant of framed links. The invariant of oriented links
is given by multiplying by $q^{4t\mathrm{writhe}}$. The writhe of
the closure of $\sigma^k$ is $k$.

\begin{ex}
For the Hopf link the invariant is given in Figure~\ref{fig:hopf}.
\end{ex}
\begin{figure}
\begin{multline*} [2\alpha+2\beta+2\gamma]
\left([2\alpha+2\beta]+[2\alpha+2\gamma]+[2\beta+2\gamma]\right) \\
-\frac{[2\alpha+2\beta+\gamma][2\alpha+\beta+2\gamma][\alpha+2\beta+2\gamma]}	{[\alpha][\beta][\gamma]}
\end{multline*}
	\caption{Universal invariant of Hopf link}\label{fig:hopf}
\end{figure}

\begin{ex}
	For the trefoil the invariant is the Laurent polynomial is given in Figure~\ref{fig:trefoil}.
\end{ex}
\begin{figure}
	\begin{gather*} 
-q^{8\alpha+8\beta+8\gamma}	\\
+q^{8\alpha+8\beta+6\gamma}
+q^{8\alpha+6\beta+8\gamma}
+q^{6\alpha+8\beta+8\gamma} \\
-q^{8\alpha+6\beta+6\gamma}
-q^{6\alpha+8\beta+8\gamma}
-q^{6\alpha+6\beta+8\gamma} \\
-q^{6\alpha+4\beta+4\gamma}
-q^{4\alpha+6\beta+4\gamma}
-q^{4\alpha+4\beta+6\gamma} \\
+q^{6\alpha+4\beta+2\gamma}
+q^{6\alpha+2\beta+4\gamma}
+q^{4\alpha+6\beta+2\gamma}
+q^{4\alpha+2\beta+6\gamma}
+q^{2\alpha+6\beta+4\gamma}
+q^{2\alpha+4\beta+6\gamma} \\
+3q^{4\alpha+4\beta+4\gamma} \\
-q^{4\alpha+4\beta+2\gamma}
-q^{4\alpha+2\beta+4\gamma}
-q^{2\alpha+4\beta+4\gamma} \\
+q^{4\alpha+2\beta+2\gamma}
+q^{2\alpha+4\beta+2\gamma}
+q^{2\alpha+2\beta+4\gamma} \\
-q^{4\alpha}
-q^{4\beta}
-q^{4\gamma} \\
-q^{2\alpha+2\beta}
-q^{2\alpha+2\gamma}
-q^{2\beta+2\gamma} \\
+q^{2\alpha}
+q^{2\beta}
+q^{2\gamma} \\
+q^{2\alpha-2\beta}
+q^{2\alpha-2\gamma}
+q^{2\beta-2\gamma}
+q^{-2\alpha+2\beta}
+q^{-2\alpha+2\gamma}
+q^{-2\beta+2\gamma} \\
-2\\
+q^{-4\alpha-4\beta}
+q^{-4\alpha-4\gamma}
+q^{-4\beta-4\gamma} \\
+q^{-4\alpha-2\beta-2\gamma}
+q^{-2\alpha-4\beta-2\gamma}
+q^{-2\alpha-2\beta-4\gamma} \\
+q^{-4\alpha-4\beta-4\gamma}
	\end{gather*}
	\caption{Universal invariant of trefoil}\label{fig:trefoil}
\end{figure}

These link polynomials are interpreted as the expectation value of observables in Chern-Simons theory. For a knot, the observable is the trace of the holonomy around the curve of a connection.
From this point of view, these formulae are a non-perturbative evaluation of these expectation values. The perturbative evaluation of these expectation values are the universal Chern-Simons formulae in \cite{MR3118578} and \cite{MR3006906}. It would be interesting to relate these two approaches.

\section{Extended Vogel plane}\label{sec:planex}
In this section we extend the Vogel plane. The extra coordinate is $\tau$.
The inclusion of the Vogel plane is given by the condition $\tau=t$.
This also corresponds to regarding a group $G$ as the symmetric space
$G\times G^{\mathrm{op}}/G$ where the involution on $G\times G^{\mathrm{op}}$
is given by $(g,h)\mapsto (h^{-1},g^{-1})$ and the inclusion of $G$ is given by
$g\mapsto (g,g^{-1})$.

The trivalent graphs are modified by introducing a second type of edge
and requiring that all vertices are of the form
\begin{equation}\begin{tikzpicture}
\fill[color=blue!20] (0,0) rectangle  (2,1);
\draw[vector] (0.5,1) to [out=270,in=120] (1,0.5);
\draw[vector] (1.5,1) to [out=270,in=60] (1,0.5);
\draw[adjoint] (1,.5) -- (1,0);
\end{tikzpicture}\end{equation}

The relation~\eqref{defd} is unchanged. The relation~\eqref{tad}
is replaced by
\begin{equation}\begin{tikzpicture}
\draw[vector] (0,0) circle(0.5);
\draw[adjoint] (0,0.5) -- (0,1);
\path (1.5,0) node[black] {$\displaystyle = 0$};
\begin{pgfonlayer}{background}
\fill[fill=blue!30] (0,0) circle (1cm);
\end{pgfonlayer}
\end{tikzpicture}\end{equation}

The anti-symmetry relation,~\eqref{AS}, is replaced by
\begin{equation}\label{ASx}\begin{tikzpicture}
\draw[vector] (0,-0.4) to [out=120,in=270] (-0.3,0) to [out=90,in=-135] (45:1);
\draw[vector] (0,-0.4) to [out=60,in=270] (0.3,0) to [out=90,in=-45] (135:1);
\draw[adjoint] (0,-0.4) -- (0,-1);
\path (1.5,0) node[black] {$\displaystyle =-$};
\draw[vector,xshift=3cm] (0,0) -- (45:1);
\draw[vector,xshift=3cm] (0,0) -- (135:1);
\draw[adjoint,xshift=3cm] (0,0) -- (0,-1);
\begin{pgfonlayer}{background}
\fill[fill=blue!30] (0,0) circle (1cm);
\fill[fill=blue!30] (3,0) circle (1cm);
\end{pgfonlayer}
\end{tikzpicture}\end{equation}

The Jacobi relation,~\eqref{IHX}, is replaced by
\begin{equation}\begin{tikzpicture}
\draw[vector]  (0.2,0) --  (45:1);
\draw[vector] (-0.2,0) -- (135:1);
\draw[vector] (-0.2,0) -- (225:1);
\draw[vector]  (0.2,0) -- (315:1);
\draw[adjoint] (-0.2,0) -- (0.2,0);
\path (1.5,0) node[black] {$\displaystyle =$};
\draw[vector,xshift=3cm] (0,0.2) -- (45:1);
\draw[vector,xshift=3cm] (0,0.2) -- (135:1);
\draw[vector,xshift=3cm] (0,-0.2) -- (225:1);
\draw[vector,xshift=3cm] (0,-0.2) -- (315:1);
\draw[adjoint,xshift=3cm] (0,-0.2) -- (0,0.2);
\path (4.5,0) node[black] {$\displaystyle +$};
\draw[vector,xshift=6cm] (-0.3,-0.3) -- (0.2,0.2) -- (45:1);
\draw[vector,xshift=6cm] (0.2,-0.2) -- (-0.2,0.2) -- (135:1);
\draw[vector,xshift=6cm] (-0.2,-0.2) -- (225:1);
\draw[vector,xshift=6cm] (0.2,-0.2) -- (315:1);
\draw[adjoint,xshift=6cm] (-0.3,-0.3) -- (0.3,-0.3);
\begin{pgfonlayer}{background}
\fill[fill=blue!30] (0,0) circle (1cm);
\fill[fill=blue!30] (3,0) circle (1cm);
\fill[fill=blue!30] (6,0) circle (1cm);
\end{pgfonlayer}
\end{tikzpicture}\end{equation}

The following two relations replace~\eqref{digon}
\begin{equation}\label{digonx1}\begin{tikzpicture}
\draw[vector] (0,0) circle(0.5);
\draw[adjoint] (0,0.5) -- (0,1);
\draw[adjoint] (0,-0.5) -- (0,-1);
\path (1.5,0) node[black] {$\displaystyle = 2\tau$};
\draw[adjoint,xshift=3cm] (90:1) -- (270:1);
\begin{pgfonlayer}{background}
\fill[fill=blue!30] (0,0) circle (1cm);
\fill[fill=blue!30] (3,0) circle (1cm);
\end{pgfonlayer}
\end{tikzpicture}\end{equation}

\begin{equation}\label{digonx2}
\begin{tikzpicture}
\fill[fill=blue!30]
(0,0) circle (1cm);
\draw[vector] (0,0.5) arc (90:270:0.5);
\draw[adjoint] (0,-0.5)arc (-90:90:0.5);
\draw[vector] (0,0.5) -- (0,1);
\draw[vector] (0,-0.5) -- (0,-1);
\end{tikzpicture}
\raisebox{0.9cm}{$\,= (\tau+t)\,$}
\begin{tikzpicture}
\fill[fill=blue!30] (0,0) circle (1cm);
\draw[vector] (90:1) -- (270:1);
\end{tikzpicture}
\raisebox{0.9cm}{$\,=\,$}
\begin{tikzpicture}
\fill[fill=blue!30] (0,0) circle (1cm);
\draw[adjoint] (0,0.5) arc (90:270:0.5);
\draw[vector] (0,-0.5)arc (-90:90:0.5);
\draw[vector] (0,0.5) -- (0,1);
\draw[vector] (0,-0.5) -- (0,-1);
\end{tikzpicture}
\end{equation}

The following replaces~\eqref{triV}
\begin{equation}
\begin{tikzpicture}
\draw[vector] (30:0.5) -- (30:1);
\draw[vector] (150:0.5) -- (150:1);
\draw[adjoint] (270:0.5) -- (270:1);
\draw[vector] (150:0.5) -- (270:0.5) -- (30:0.5);
\draw[adjoint] (30:0.5) -- (150:0.5);
\path (1.5,0) node[black] {$\displaystyle = \tau$};
\draw[vector,xshift=3cm] (0,0) -- (30:1);
\draw[vector,xshift=3cm] (0,0) -- (150:1);
\draw[adjoint,xshift=3cm] (0,0) -- (270:1);
\begin{pgfonlayer}{background}
\fill[fill=blue!30] (0,0) circle (1);
\fill[fill=blue!30] (3,0) circle (1cm);
\end{pgfonlayer}
\end{tikzpicture}\end{equation}

The following replaces~\eqref{triX}
\begin{equation}\begin{tikzpicture}
\draw[vector] (30:0.5) -- (30:1);
\draw[vector] (150:0.5) -- (150:1);
\draw[black] (270:0.5) -- (270:1);
\draw[vector] (150:0.5) -- (270:0.5) -- (30:0.5);
\draw[adjoint] (30:0.5) -- (150:0.5);
\path (1.5,0) node[black] {$\displaystyle = 0$};
\path (0,-1) node[black] {$X$};
\begin{pgfonlayer}{background}
\fill[fill=blue!30] (0,0) circle (1);
\end{pgfonlayer}
\end{tikzpicture}\end{equation}

The two string centraliser algebra has dimension six and we assume the following are a basis:
\[ 1, U, K, H, X, H^2 \]
These correspond to the diagrams
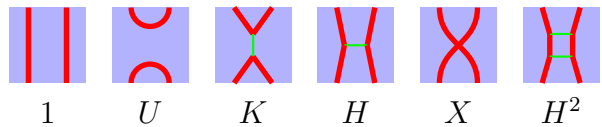
\begin{center}
	\begin{tabular}{cccccc}
		\begin{tikzpicture}[red,line width=2pt]
		\fill[fill=blue!30] (0,0) rectangle (1,1);
		\draw (0.25,1) -- (0.25,0);
		\draw (0.75,1) -- (0.75,0);
		\end{tikzpicture} &
		\begin{tikzpicture}[red,line width=2pt]
		\fill[fill=blue!30] (0,0) rectangle (1,1);
		\draw (0.75,0) arc (0:180:0.25);
		\draw (0.75,1) arc (0:-180:0.25);
		\end{tikzpicture} &
		\begin{tikzpicture}[red,line width=2pt]
		\fill[fill=blue!30] (0,0) rectangle (1,1);
		\draw (0.25,1) -- (0.5,0.65);
		\draw (0.75,1) -- (0.5,0.65);
		\draw (0.5,0.65)[adjoint] -- (0.5,0.35);
		\draw (0.25,0) -- (0.5,0.35);
		\draw (0.75,0) -- (0.5,0.35);
		\end{tikzpicture} &
		\begin{tikzpicture}[red,line width=2pt]
		\fill[fill=blue!30] (0,0) rectangle (1,1);
		\draw (0.25,1) -- (0.35,0.5);
		\draw (0.75,1) -- (0.65,0.5);
		\draw (0.65,0.5)[adjoint] -- (0.35,0.5);
		\draw (0.25,0) -- (0.35,0.5);
		\draw (0.75,0) -- (0.65,0.5);
		\end{tikzpicture} &
		\begin{tikzpicture}[red,line width=2pt]
		\fill[fill=blue!30] (0,0) rectangle (1,1);
		\draw (0.25,1) .. controls (0.25,0.5) and (0.75,0.5) .. (0.75,0);
		\draw (0.75,1) .. controls (0.75,0.5) and (0.25,0.5) .. (0.25,0);
		\end{tikzpicture} &
		\begin{tikzpicture}[red,line width=2pt]
		\fill[fill=blue!30] (0,0) rectangle (1,1);
		\draw (0.35,0.65) -- (0.35,0.35);
		\draw (0.65,0.65) -- (0.65,0.35);
		\draw (0.35,0.35)[adjoint] -- (0.65,0.35);
		\draw (0.35,0.65)[adjoint] -- (0.65,0.65);
		\draw (0.25,1) -- (0.35,0.65);
		\draw (0.75,1) -- (0.65,0.65);
		\draw (0.25,0) -- (0.35,0.35);
		\draw (0.75,0) -- (0.65,0.35);
		\end{tikzpicture} \\
		$1$ & $U$ & $K$ & $H$ & $X$ & $H^2$
	\end{tabular}
\end{center}

Then we have the following analogue of~\eqref{eq:H3}.
\begin{equation}
H^3=t\,H^2 -s\,H + \left(\frac s2 + \frac{\tau }{2}(\tau -t)\right) K%
 + \frac p2 (U+X-2)
\end{equation}

Then we also have an additional relation
\begin{multline}\label{eq:sq}
\begin{tikzpicture}
\fill[fill=blue!30] (0,0) circle (1cm);
\draw[vector] (45:0.5) -- (45:1);
\draw[vector] (135:0.5) -- (135:1);
\draw[vector] (225:0.5) -- (225:1);
\draw[vector] (315:0.5) -- (315:1);
\draw[vector] (45:0.5) -- (135:0.5);
\draw[adjoint] (135:0.5) -- (225:0.5);
\draw[vector] (225:0.5) -- (315:0.5);
\draw[adjoint] (315:0.5) --(45:0.5);
\end{tikzpicture}
\raisebox{0.9cm}{\,$-$\,}
\begin{tikzpicture}
\fill[fill=blue!30] (0,0) circle (1cm);
\draw[vector] (45:0.5) -- (45:1);
\draw[vector] (135:0.5) -- (135:1);
\draw[vector] (225:0.5) -- (225:1);
\draw[vector] (315:0.5) -- (315:1);
\draw[adjoint] (45:0.5) -- (135:0.5);
\draw[vector] (135:0.5) -- (225:0.5);
\draw[adjoint] (225:0.5) -- (315:0.5);
\draw[vector] (315:0.5) --(45:0.5);
\end{tikzpicture} \\
\raisebox{0.9cm}{\,$=(\tau-t)$\,$\Bigg($}
\begin{tikzpicture}
\fill[fill=blue!30] (0,0) circle (1cm);
\draw[vector]  (0.2,0) --  (45:1);
\draw[vector] (-0.2,0) -- (135:1);
\draw[vector] (-0.2,0) -- (225:1);
\draw[vector]  (0.2,0) -- (315:1);
\draw[adjoint] (-0.2,0) -- (0.2,0);
\end{tikzpicture}
\raisebox{0.9cm}{\,$-$\,}
\begin{tikzpicture}
\fill[fill=blue!30] (0,0) circle (1cm);
\draw[vector] (0,0.2) -- (45:1);
\draw[vector] (0,0.2) -- (135:1);
\draw[vector] (0,-0.2) -- (225:1);
\draw[vector] (0,-0.2) -- (315:1);
\draw[adjoint] (0,-0.2) -- (0,0.2);
\end{tikzpicture}
\raisebox{0.9cm}{$\Bigg)$}
\end{multline}

The dimension formulae which give~\eqref{dimV} under the specialisation
$\tau=t$ are:
\begin{align}\label{dimVx}
\delta &= (-\tau^3-2\tau^2t-\tau t^2-s\tau-s t+p)/p \\
&= -\frac{(\tau+t-\alpha)(\tau+t-\beta)(\tau+t-\gamma)}{\alpha\,\beta\,\gamma} \\
\dim(\fg) &=
\frac{(\tau+t)(\tau^3+2t\tau^2+(s+t^2)\tau+st-p)}{p\tau} \\
&=
-1/2\,{\frac { \left( \beta+\gamma+\tau \right)  \left( \alpha+\gamma+\tau
		\right)  \left( \alpha+\beta+\tau \right)  \left( \tau+t
		\right) }{\alpha\,\beta\,\gamma\,\tau}} \\
&= \delta\, \left( \frac{\tau+t}{2\tau} \right)
\end{align}
The dimension formulae which give~\eqref{dimX} under the specialisation
$\tau=t$ is:
\begin{equation}\label{dimXx}
\dim(X) =
-\delta/2\,{\frac { \left( \gamma+\tau \right)  \left( \beta+\tau \right) 
		\left( \tau+\alpha \right)  \left( \tau+t \right) }{
		\alpha\,\beta\,\gamma\,\tau}} 
\end{equation}
The dimension formulae which give~\eqref{dimY} under the specialisation
$\tau=t$ is:
\begin{equation}\label{dimYx}
1/2\,{\frac { \left( \alpha+\gamma+\tau \right)  \left( \alpha+\beta+\tau
		\right)  \left( \gamma+\tau \right)  \left( \beta+\tau \right)  \left( \tau+
		t \right)  \left( \tau+t-3\,\alpha \right) 
	}{{\alpha}^{2}\beta\,\gamma\, \left( \alpha-\gamma \right)  \left( 
	\alpha-\beta \right) }}
\end{equation}

\begin{prop}\label{prop:casx}
	The values of the Casimir are given by
\begin{center}
	\begin{tabular}{cccccc}
		$\fg$ & $V$ & $X$ & $Y_\alpha$ & $Y_\beta$ & $Y_\gamma$ \\ \hline
		$t$ & $(\tau+t)/2$ & $(\tau+t)$ & $(\tau+t-\alpha)$ & $(\tau+t-\beta)$ & $(\tau+t-\gamma)$ \\
		
	\end{tabular}
\end{center}	
\end{prop}

\begin{proof} These values were found from the properties that the specialisation $\tau=t$ gives
Proposition~\ref{prop:cas} and that it fits with the examples in Section~\ref{sec:ex}.
\end{proof}

\section{Examples}\label{sec:ex}
In these three examples we have a series. These involve classical
groups and the dimension formulae are given in \cite{MR0271271}.

In the first two examples $V\cong Y_\alpha$
since $V$ has a symmetric multiplication.

These two series are in fact a single line since they are interchanged by
the change of variable $n\leftrightarrow -n$; this is discussed in
\cite{MR2858070}.

\subsection{AI}
Take $\fg=\fso(n)$. For $n=2r+1$ this has type $B_r$ and for $n=2r$ this has type $D_r$. In both cases take $V$ to have highest weight $2\omega_1$, so
$V\oplus 1$ is the symmetric square of the vector representation. Then we have $\fso(n)\oplus V\fsl(n)$ has type $A_{2r}$.

The symmetric multiplication can be seen by regarding $V$ as the special Jordan
algebra of trace-free symmetric matrices.

Then indexing irreducible representations
by highest weight and by partitions we have
\begin{center}
	\begin{tabular}{cccccc}
$\fg$ & $V$ & $X$ & $Y_\alpha$ & $Y_\beta$ & $Y_\gamma$ \\ \hline
$\omega_2$ &  $2\omega_1$ & $2\omega_1+\omega_2$ & $2\omega_1$ & $2\omega_2$ & $4\omega_1$ \\
$[1,1]$ & $[2]$ & $[3,1]$ & $[2]$ & $[2,2]$ & $[4]$ 
	\end{tabular}
\end{center}
Then we also have the following dimension formulae
\begin{align*}
	\dim \fg &= \frac 12 n(n-1) \\
	\dim V &= \frac 12 (n-1)(n+2) \\
	\dim X &= \frac18 (n-2)(n-1)(n+1)(n+4) \\
	\dim Y_\alpha &= \frac 12 (n-1)(n+2) \\
	\dim Y_\beta &= \frac1{12} (n-3)n(n+1)(n+2)\\
	\dim Y_\gamma &= \frac1{24} (n-1)n(n+1)(n+6)\\
\end{align*}
The Casimirs are given by
\begin{center}
	\begin{tabular}{cccccc}
		$\fg$ & $V$ & $X$ & $Y_\alpha$ & $Y_\beta$ & $Y_\gamma$ \\ \hline
		$2n-4$ & $2n$ & $4n$ & $2n$ & $4n+8$ & $4n-4$
	\end{tabular}
\end{center}

This gives
\begin{center}
	\begin{tabular}{ccccc}
		$\tau$ & $t$ & $\alpha$ & $\beta$ & $\gamma$ \\ \hline
		$n+2$ & $n-2$ & $n$ & 2 & -4 
	\end{tabular}
\end{center}
\subsection{AII}
Take $\fg=\fsp(2r)$ of type $C_r$ and $V$ to have highest
weight $\omega_2$. Then $\fg\oplus V$ has type $A_{2r-1}$.

We also have $\fsl(2r)\cong \fsp(2r)\oplus V$ where $V\oplus 1$
is the exterior square of the vector representation. 

Then indexing irreducible representations by highest weight and 
by partitions we have
\begin{center}
	\begin{tabular}{cccccc}
		$\fg$ & $V$ & $X$ & $Y_\alpha$ & $Y_\beta$ & $Y_\gamma$ \\ \hline
		$2\omega_1$ & $\omega_2$ & $\omega_1+\omega_3$ & $\omega_2$ & $2\omega_2$ & $\omega_4$ \\
		$[2]$ & $[1,1]$ & $[2,1,1]$ & $[1,1]$ & $[2,2]$ & $[1^4]$ 
	\end{tabular}
\end{center}
Putting $n=2r$, we also have the following dimension formulae
\begin{align*}
	\dim \fg &= \frac 12 n(n+1) \\
	\dim V &= \frac 12 (n-2)(n+1) \\
	\dim X &= \frac18 (n-4)(n-2)(n-1)(n+1) \\
	\dim Y_\alpha &= \frac 12 (n-2)(n+1) \\
	\dim Y_\beta &= \frac1{12} n(n-2)(n+3)(n-1)\\
	\dim Y_\gamma &= \frac1{24} n(n-6)(n-1)(n+1)
\end{align*}

The Casimirs are given by
\begin{center}
	\begin{tabular}{cccccc}
		$\fg$ & $V$ & $X$ & $Y_\alpha$ & $Y_\beta$ & $Y_\gamma$ \\ \hline
		$2n+4$ & $2n$ & $4n$ & $2n$ & $4n+4$ & $4n-8$
	\end{tabular}
\end{center}

This gives
\begin{center}
	\begin{tabular}{ccccc}
		$\tau$ & $t$ & $\alpha$ & $\beta$ & $\gamma$ \\ \hline
		$n-2$ & $n+2$ & $n$ & -2 & 4 
	\end{tabular}
\end{center}
It is an open problem to give an interpretation for $n$ odd.
Presumably this is related to the odd symplectic groups.
\subsection{BDI}
Take $\fg=\fso(2n+1)$ of type $B_n$ and $V$ to have highest
weight $\omega_1$. Then $\fg\oplus V$ has type $D_{n+1}$.
Then $X$, $Y_\beta$ and $Y_\gamma$ are all zero. This is reminiscent
of the $SL(2)$ line in the Vogel plane.
Then indexing irreducible representations by highest weight and
by partitions we have
\begin{center}
	\begin{tabular}{ccc}
		$\fg$ & $V$ & $Y_\alpha$  \\ \hline
		$\omega_2$ & $\omega_1$ & $2\omega_1$ \\
		$[1,1]$ & $[1]$ & $[2]$ 
	\end{tabular}
\end{center}
Then we also have the following dimension formulae
\begin{align*}
	\dim \fg &= \frac 12 n(n-1) \\
	\dim V &= n \\
	\dim Y_\alpha &= \frac 12 (n-1)(n+2) \\
\end{align*}

This gives
\begin{center}
	\begin{tabular}{ccccc}
		$\tau$ & $t$ & $\alpha$ & $\beta$ & $\gamma$ \\ \hline
		$1$ & $n-2$ & $-1$ & $\beta$ & $n-\beta-1$
	\end{tabular}
\end{center}

\subsection{Exceptional series}
These examples can be considered as an analogue of the exceptional series
since we have $Y(\gamma) \cong 0$. The equation of this plane is
\[ \tau+\alpha+\beta-2\gamma=0\]
\subsubsection{FII} This is based on $F_4 \cong B_4 \oplus V$.
The highest weights, dimensions and Casimirs are given by
\begin{center}
	\begin{tabular}{ccccc}
		$\fg$ & $V$ & $X$ & $Y_\alpha$ & $Y_\beta$ \\ \hline
		$\omega_2$ & $\omega_4$ & $\omega_3$ & $\omega_1$ & $2\omega_4$ \\
		36 & 16 & 84 & 9 & 126 \\
		14 & 9 & 18 & 8 & 20
	\end{tabular}
\end{center}
This gives
\begin{center}
	\begin{tabular}{ccccc}
		$\tau$ & $t$ & $\alpha$ & $\beta$ & $\gamma$ \\ \hline
		$2$ & $7$ & $-1$ & $3$ & $5$ \\
	\end{tabular}
\end{center}
\subsubsection{EIV} This is based on $E_6 \cong F_4 \oplus V$.
In this example we again have $V\cong Y_\alpha$ and the symmetric
multiplication arises from regarding $V$ as an exceptional Jordan algebra.

The highest weights, dimensions and Casimirs are given by
\begin{center}
	\begin{tabular}{ccccc}
		$\fg$ & $V$ & $X$ & $Y_\alpha$ & $Y_\beta$ \\ \hline
		$\omega_1$ & $\omega_4$ & $\omega_3$ & $\omega_4$ & $2\omega_4$ \\
		52 & 26 & 273 & 26 & 324 \\
		18 & 12 & 24 & 12 & 26
	\end{tabular}
\end{center}
This gives
\begin{center}
	\begin{tabular}{ccccc}
		$\tau$ & $t$ & $\alpha$ & $\beta$ & $\gamma$ \\ \hline
		$3$ & $9$ & $6$ & $-1$ & $4$ \\
	\end{tabular}
\end{center}
\subsubsection{EI} This is based on $E_6 \cong C_4 \oplus V$.
The highest weights, dimensions and Casimirs are given by
\begin{center}
	\begin{tabular}{ccccc}
		$\fg$ & $V$ & $X$ & $Y_\alpha$ & $Y_\beta$ \\ \hline
		$2\omega_1$ & $\omega_4$ & $2\omega_3$ & $2\omega_4$ & $2\omega_2$ \\
		36 & 42 & 825 & 594 & 308 \\
		10 & 12 & 24 & 28 & 18
	\end{tabular}
\end{center}
This gives
\begin{center}
	\begin{tabular}{ccccc}
		$\tau$ & $t$ & $\alpha$ & $\beta$ & $\gamma$ \\ \hline
		$7$ & $5$ & $-2$ & $3$ & $4$ \\
	\end{tabular}
\end{center}
\subsubsection{EV} This is based on $E_7 \cong A_7 \oplus V$.
The highest weights, dimensions and Casimirs are given by
\begin{center}
	\begin{tabular}{ccccc}
		$\fg$ & $V$ & $X$ & $Y_\alpha$ & $Y_\beta$ \\ \hline
$\omega_1+\omega_7$ & $\omega_4$ & $\omega_3+\omega_5$ & $2\omega_4$ & $\omega_2+\omega_6$ \\
		63 & 70 & 2352 & 1764 & 720 \\
		16 & 18 & 36 & 40 & 28
	\end{tabular}
\end{center}
This gives
\begin{center}
	\begin{tabular}{ccccc}
		$\tau$  & $t$ & $\alpha$ & $\beta$ & $\gamma$ \\ \hline
		$5$  & $4$ & $-1$ & $2$ & $3$ \\
	\end{tabular}
\end{center}
\subsubsection{EVIII} This is based on $E_8 \cong D_8 \oplus V$.
The highest weights, dimensions and Casimirs are given by
\begin{center}
	\begin{tabular}{ccccc}
		$\fg$ & $V$ & $X$ & $Y_\alpha$ & $Y_\beta$ \\ \hline
		$\omega_2$ & $\omega_8$ & $\omega_6$ & $2\omega_8$ & $\omega_4$ \\
		120 & 128 & 8008 & 6435 & 1820 \\
		28 & 30 & 60 & 64 & 48 
	\end{tabular}
\end{center}
This gives
\begin{center}
	\begin{tabular}{ccccc}
		$\tau$  & $t$ & $\alpha$ & $\beta$ & $\gamma$ \\ \hline
		$8$ & $7$ & $-1$ & $3$ & $5$ \\
	\end{tabular}
\end{center}

\section{Quantisation of extended Vogel plane}\label{sec:planexq}
In this section we give the quantisation of the extended Vogel plane.
This is a finitely presented ribbon category for which the presentation is a $q$-analogue of the presentation of the extended Vogel plane.

We introduce an additional indeterminate $q^\tau$. We extend the notation of \eqref{eq:qexp}
by putting
\begin{equation}
q^{n_\tau\tau+n_\alpha\alpha+n_\beta\beta+n_\gamma\gamma}=
(q^{\tau})^{n_\tau}.(q^{\alpha})^{n_\alpha}.(q^\beta)^{n_\beta}.(q^\gamma)^{n_\gamma}
\end{equation}
for integers $n_\tau,n_\alpha,n_\beta,n_\gamma$.
We also use the notation of \eqref{eq:qint} by putting
\begin{equation}
[n_\tau\tau+n_\alpha\alpha+n_\beta\beta+n_\gamma\gamma]=
q^{n_\tau\tau+n_\alpha\alpha+n_\beta\beta+n_\gamma\gamma}-
q^{-n_\tau\tau-n_\alpha\alpha-n_\beta\beta-n_\gamma\gamma}
\end{equation}

The rational functions in this section then have similar properties to the rational
functions in section \ref{sec:planeq}. In particular, they have a classical limit.

\subsection{Exceptional plane}

The main result of \cite{MR1815266} is that five dimensional representations of
the three string braid group are determined by their eigenvalues. The exceptional
plane gives a five dimensional representation with eigenvalues given by
Table~\ref{table:brqx}. The matrices representing the generators $\sigma_1$ and
$\sigma_2$ can be given explicitly. The braid matrix, $\sigma_1$, is given in
Figure~\ref{fig:br5}. This matrix has been scaled so that the determinant is 1.

\begin{figure}
	\begin{equation*}
	\left[ \begin {array}{ccccc}  q^{-4\gamma}&0&0&0&0\\
	\noalign{\medskip}-\frac { [2\alpha+2\beta]}{[\alpha+\beta]}
	\frac { [2\beta+2\gamma]}{[\beta+\gamma]}q^{\alpha+\beta-2\gamma}
	&-q^{2\alpha+2\beta-2\gamma}&0&0&0\\
	\noalign{\medskip}{\frac { [2\beta+2\gamma]}{[\beta+\gamma]}q^{\beta-\gamma}}
	& q^{\beta}
	&  q^{-2\alpha+2\gamma} &0&0\\
	\noalign{\medskip}\frac { [2\alpha+2\beta]}{[\alpha+\beta]}
	\frac { [2\beta+2\gamma]}{[\beta+\gamma]}q^{-\alpha+\beta+\gamma}
	& (\sigma_1)_{4,2}		
	&(\sigma_1)_{4,3}
	&-  q^{2\gamma}  &0\\
	\noalign{\medskip} q^{2\beta+2\gamma} 
	&  q^{\beta+2\gamma}
	&\frac {[3\alpha-3\gamma]}{[\alpha-\gamma]}q^{-2\gamma}
	&- q^{-\beta+2\gamma}
	& q^{-2\beta+2\gamma}\end {array}
	\right] 
	\end{equation*}
	\begin{equation*}
	(\sigma_1)_{4,2} = q^{2\beta+2\gamma} +  1+  q^{-2\alpha+2\gamma} \qquad
	(\sigma_1)_{4,3} = \frac {[3\alpha-3\gamma]}{[\alpha-\gamma]}
	\frac{[2\alpha+2\beta]}{[\alpha+\beta]}q^{-2\alpha+2\gamma}
	\end{equation*}
	\caption{Braid matrix}\label{fig:br5}
\end{figure}

Also, the inverse of $\sigma_1$ is obtained by applying the
bar involution to each entry. The bar involution is $q^\alpha\leftrightarrow q^{-\alpha}$,
$q^\beta\leftrightarrow q^{-\beta}$, $q^\gamma\leftrightarrow q^{-\gamma}$.

Define the involution $P$ by $P_{ij}=\delta_{i+j,6}$. Then put $\sigma_2=P\sigma_1P$.
These satisfy the relation $\sigma_1\sigma_2\sigma_1=P$ and this implies
that the braid relation $\sigma_1\sigma_2\sigma_1=\sigma_2\sigma_1\sigma_2$.

Since this representation is given explicitly the coefficients we are interested in can
be computed directly. In particular, to compute the quantum dimensions, we first take
the spectral decomposition of $sigma_1$, so
\begin{equation}
\sigma_1 = \sum_{i} \rho^{(i)}E_1^{(i)}
\end{equation}
The idempotent, $E_1$, corresponding to the eigenvalue $q^{2\tau+2t}$ projects to the trivial
representation. Defining $E_2$ similarly we compute $\delta$ from $E_1E_2E_1=\delta^2E_1$.
The elements $U_i=\delta E_i$ then satisfy the Temperley-Lieb relations. The quantum 
dimensions can then be computed using
\begin{equation}
\delta U_2 E_1^{(i)} U_2 = \dim_q V^{(i)} U_2
\end{equation}
\subsection{Extended plane}
The eigenvalues of the braiding $\sigma$ are given in Table~\ref{table:brqx}.
This is a consequence of Proposition~\ref{prop:cas}.
\begin{figure}
	\begin{tabular}{cc|cccc}
		$\fg$ & $X$ & $1$ & $Y(\alpha)$ & $Y(\beta)$ & $Y(\gamma)$ \\ \hline
		$-q^{2\tau}$ & $-1$ & $q^{2\tau+2t}$ & $q^{2\alpha}$ & $q^{2\beta}$ & $q^{2\gamma}$ 
	\end{tabular}
	\caption{Eigenvalues of braid generator}\label{table:brqx}
\end{figure}
This gives Table~\ref{table:brq} under the specialisation $\tau=t$.

The antisymmetry relations~\eqref{ASq} and~\eqref{ASq} are replaced by
\begin{equation}\label{ASqx}
\begin{tikzpicture}[red,line width=2pt]
\fill[fill=blue!30] (0,0) circle (1cm);
\draw[vector] (0,-0.4) to [out=120,in=270] (-0.3,0) to [out=90,in=-135] (45:1);
\fill[blue!30] (0,0.35) circle(0.2);
\draw[vector] (0,-0.4) to [out=60,in=270] (0.3,0) to [out=90,in=-45] (135:1);
\draw[adjoint] (0,-0.4) -- (0,-1);
\end{tikzpicture}
\raisebox{0.9cm}{$\,=q^{2\tau}\,$}
 \begin{tikzpicture}[red,line width=2pt]
 \fill[fill=blue!30] (0,0) circle (1cm);
\draw[vector] (0,0) -- (45:1);
\draw[vector] (0,0) -- (135:1);
\draw[adjoint] (0,0) -- (270:1);
\end{tikzpicture}
\qquad
\begin{tikzpicture}
\fill[fill=blue!30] (0,0) circle (1cm);
\draw[vector] (0,-0.4) to [out=60,in=270] (0.3,0) to [out=90,in=-45] (135:1);
\fill[blue!30] (0,0.35) circle(0.2);
\draw[vector] (0,-0.4) to [out=120,in=270] (-0.3,0) to [out=90,in=-135] (45:1);
\draw[adjoint] (0,-0.4) -- (0,-1);
\end{tikzpicture}
\raisebox{0.9cm}{$\,=q^{-2\tau}\,$}
\begin{tikzpicture}[red,line width=2pt]
\fill[fill=blue!30] (0,0) circle (1cm);
\draw[vector] (0,0) -- (45:1);
\draw[vector] (0,0) -- (135:1);
\draw[adjoint] (0,0) -- (270:1);
\end{tikzpicture}
\end{equation}

The inner product relation~\eqref{eq:inn} is replaced by
\begin{equation}
\begin{tikzpicture}[vector]
\fill[fill=blue!30] (0,0) circle (1cm);
\draw (-0.3,0) to [out=90,in=225] (45:1);
\fill[blue!30] (0,0.3) circle(0.2);
\draw (0.3,0) to [out=90,in=315] (135:1);
\draw (0.3,0) arc (360:180:0.3);
\end{tikzpicture}
\raisebox{0.9cm}{$\,=q^{2\tau+2t}\,$}
\begin{tikzpicture}[red,line width=2pt]
\fill[fill=blue!30] (0,0) circle (1cm);
\draw (45:1) to [out=225,in=0] (0,0.1) to [out=180,in=315] (135:1);
\end{tikzpicture}
\qquad
\begin{tikzpicture}[red,line width=2pt]
\fill[fill=blue!30] (0,0) circle (1cm);
\draw (0.3,0) to [out=90,in=315] (135:1);
\fill[blue!30] (0,0.3) circle(0.2);
\draw (-0.3,0) to [out=90,in=225] (45:1);
\draw (0.3,0) arc (360:180:0.3);
\end{tikzpicture}
\raisebox{0.9cm}{$\,=q^{-2\tau-2t}\,$}
\begin{tikzpicture}[red,line width=2pt]
\fill[fill=blue!30] (0,0) circle (1cm);
\draw (45:1) to [out=225,in=0] (0,0.1) to [out=180,in=315] (135:1);
\end{tikzpicture}
\end{equation}

The relation~\eqref{digonx2} is replaced by
\begin{equation}\label{digonxq2}
\begin{tikzpicture}
\fill[fill=blue!30]
(0,0) circle (1cm);
\draw[vector] (0,0.5) arc (90:270:0.5);
\draw[adjoint] (0,-0.5)arc (-90:90:0.5);
\draw[vector] (0,0.5) -- (0,1);
\draw[vector] (0,-0.5) -- (0,-1);
\end{tikzpicture}
\raisebox{0.9cm}{$\,= \,$}
\begin{tikzpicture}
\fill[fill=blue!30] (0,0) circle (1cm);
\draw[vector] (90:1) -- (270:1);
\end{tikzpicture}
\raisebox{0.9cm}{$\,=\,$}
\begin{tikzpicture}
\fill[fill=blue!30] (0,0) circle (1cm);
\draw[adjoint] (0,0.5) arc (90:270:0.5);
\draw[vector] (0,-0.5)arc (-90:90:0.5);
\draw[vector] (0,0.5) -- (0,1);
\draw[vector] (0,-0.5) -- (0,-1);
\end{tikzpicture}
\end{equation}
This is a choice of normalisation of the trivalent vertex.

The coefficient $2\tau$ in the relation~\eqref{digonx1} is replaced by
\begin{multline}
\frac{[\alpha+\beta]}{[2\alpha+2\beta]}
\frac{[\alpha+\gamma]}{[2\alpha+2\gamma]}
\frac{[\beta+\gamma]}{[2\beta+2\gamma]} \\
\frac{[2\tau-2\alpha]}{[\tau-\alpha]}
\frac{[2\tau-2\beta]}{[\tau-\beta]}
\frac{[2\tau-2\gamma]}{[\tau-\gamma]} \\
\frac{[2\tau-t]}{[4\tau-2t]}
\frac{[2t]}{[t]}
\frac{[2\tau]}{[\tau+t]}
\end{multline}

The eigenvalues of $H$ are given by
\begin{center}
	\begin{tabular}{cccc}
		$1$ & $Y(\alpha)$ & $Y(\beta)$ & $Y(\gamma)$ \\ \hline
		$1$ & $\frac{[\alpha][\beta+\gamma][2\tau]}{[\tau][2\beta+2\gamma][\tau+t]}$ & $\frac{[\beta][\alpha+\gamma][2\tau]}{[\tau][2\alpha+2\gamma][\tau+t]}$ & $\frac{[\gamma][\alpha+\beta][2\tau]}{[\tau][2\alpha+2\beta][\tau+t]}$ 
	\end{tabular}
\end{center}

The eigenvalue of $H$ associated to $X$ is
\begin{equation}
- \frac{[\alpha][\beta][\gamma]}{[\tau+t]}
\, \frac{[\beta+\gamma]}{[2\beta+2\gamma]}\, \frac{[\alpha+\gamma]}{[2\alpha+2\gamma]}\, \frac{[\alpha+\beta]}{[2\alpha+2\beta]}\,
\frac{[2\tau]}{[\tau]}
\end{equation}
This gives~\eqref{eq:hxq} under the specialisation $q^\tau=q^t$.

The eigenvalue of $H$ associated to $\fg$ is
\begin{multline}
\frac 1{[\tau+t]}\, \frac{[2\tau-t]}{[4\tau-2t]}\, \frac{[2\tau]}{[\tau]}\,
\frac{[\alpha+\beta]}{[2\alpha+2\beta]}\,
\frac{[\alpha+\gamma]}{[2\alpha+2\gamma]}\,
\frac{[\beta+\gamma]}{[2\beta+2\gamma]} \\
\times \Bigg(-2[\tau]-[2t]+[2\alpha+2\beta]+[2\alpha+2\gamma]+[2\beta+2\gamma]
-[2\alpha]-[2\beta]-[2\gamma] \\
+[2\tau-2\alpha]+[2\tau-2\beta]+[2\tau-2\gamma]
+[2\tau+2\alpha]+[2\tau+2\beta]+[2\tau+2\gamma]
\Bigg)
\end{multline}
I have not (yet) found a way of writing this that makes it clear that
this gives~\eqref{eq:hk} under the specialisation $q^\tau=q^t$.

The dimension formula that gives~\eqref{dimVx} in the classical limit
and~\eqref{dimVq} under the specialisation $q^\tau=q^t$ is:
\begin{equation}
\delta = -\frac{[t+\tau-\alpha][t+\tau-\beta][t+\tau-\gamma]}{[\alpha][\beta][\gamma]}%
\frac{[2t][\tau]}{[t][2\tau]}
\end{equation}

The dimension formula that gives~\eqref{dimVx} in the classical limit
and~\eqref{dimVq} under the specialisation $q^\tau=q^t$ is:
\begin{multline}
\dim_q\fg =  -\frac{[t+\tau-\alpha][t+\tau-\beta][t+\tau-\gamma]}{[\alpha][\beta][\gamma]} \\
\times \frac{[2t-2\alpha]}{[t-\alpha]}\, \frac{[2t-2\beta]}{[t-\beta]} \, \frac{[2t-2\gamma]}{[t-\gamma]} \\
\times \frac{[\tau-\alpha]}{[2\tau-2\alpha]}\, \frac{[\tau-\beta]}{[2\tau-2\beta]}\, \frac{[\tau-\gamma]}{[2\tau-2\gamma]} \\
\times \frac{[\tau+t][4\tau-2t][\tau]}{[2\tau][2\tau-t][2\tau]}
\end{multline}
The dimension formula that gives~\eqref{dimXx} in the classical limit
and~\eqref{dimXq} under the specialisation $q^\tau=q^t$ is:
\begin{multline}
\dim_q X =
\frac{[\tau+\alpha][\tau+\beta][\tau+\gamma]}{[\alpha][\beta][\gamma]} \\
\times \frac{[t+\tau+\alpha][t+\tau+\beta][t+\tau+\gamma]}{[2\alpha][2\beta][2\gamma]} \\
\times \frac{[2t-2\alpha]}{[t-\alpha]}\, \frac{[2t-2\beta]}{[t-\beta]} \, \frac{[2t-2\gamma]}{[t-\gamma]}
\times \frac{[\tau][2t][\tau+t]}{[2\tau][t][2\tau]}
\end{multline}
or, equivalently,
\begin{multline}
	\dim_q X =
	-\delta\frac{[\tau+\alpha][\tau+\beta][\tau+\gamma]}{[2\alpha][2\beta][2\gamma]} \\
	\times \frac{[2t-2\alpha]}{[t-\alpha]}\, \frac{[2t-2\beta]}{[t-\beta]} \, \frac{[2t-2\gamma]}{[t-\gamma]}
	\times \frac{[\tau+t]}{[2\tau]}
\end{multline}
The dimension formula that gives~\eqref{dimYx} in the classical limit
and~\eqref{dimYq} under the specialisation $q^\tau=q^t$ is:
\begin{multline}
\dim_q Y(\alpha)=\\ \frac{[\tau+t][\tau+\beta][\tau+\gamma][\tau+t-\beta][\tau+t-\gamma][\tau+t-3\alpha]}{[2\alpha][\alpha][\beta][\gamma][\alpha-\beta][\alpha-\gamma]} \\
\times \frac{[2t-2\alpha][\tau-\alpha]}{[t-\alpha][2\tau-2\alpha]}
 \frac{[2t][\tau]}{[t][2\tau]}	
\end{multline}

The relation that replaces \eqref{eq:sq} is given in Figure~\ref{fig:sqx}.
Note the coefficient $c_0$ is 0 both in the classical limit and under the specialisation $q^\tau=q^t$.
\begin{figure}
\begin{multline}
\begin{tikzpicture}
\fill[fill=blue!30] (0,0) circle (1cm);
\draw[vector] (45:0.5) -- (45:1);
\draw[vector] (135:0.5) -- (135:1);
\draw[vector] (225:0.5) -- (225:1);
\draw[vector] (315:0.5) -- (315:1);
\draw[vector] (45:0.5) -- (135:0.5);
\draw[adjoint] (135:0.5) -- (225:0.5);
\draw[vector] (225:0.5) -- (315:0.5);
\draw[adjoint] (315:0.5) --(45:0.5);
\end{tikzpicture}
\raisebox{0.9cm}{\,$-$\,}
\begin{tikzpicture}
\fill[fill=blue!30] (0,0) circle (1cm);
\draw[vector] (45:0.5) -- (45:1);
\draw[vector] (135:0.5) -- (135:1);
\draw[vector] (225:0.5) -- (225:1);
\draw[vector] (315:0.5) -- (315:1);
\draw[adjoint] (45:0.5) -- (135:0.5);
\draw[vector] (135:0.5) -- (225:0.5);
\draw[adjoint] (225:0.5) -- (315:0.5);
\draw[vector] (315:0.5) --(45:0.5);
\end{tikzpicture}
\raisebox{0.9cm}{\,$=$} \\
\raisebox{0.9cm}{$c_0\Bigg($\,}
\begin{tikzpicture}[red,line width=2pt]
\fill[blue!30] (0,0) circle (1cm);
\draw (45:1) .. controls (45:0.5) and (315:0.5) .. (315:1);
\draw (135:1) .. controls (135:0.5) and (225:0.5) .. (225:1);
\end{tikzpicture}
\raisebox{0.9cm}{\,$-$\,}
\begin{tikzpicture}[red,line width=2pt]
\fill[blue!30] (0,0) circle (1cm);
\draw (45:1) .. controls (45:0.5) and (135:0.5) .. (135:1);
\draw (315:1) .. controls (315:0.5) and (225:0.5) .. (225:1);
\end{tikzpicture}
\raisebox{0.9cm}{\,$\Bigg)$} \\
\raisebox{0.9cm}{$+c_1\Bigg($\,}
\begin{tikzpicture}
\fill[fill=blue!30] (0,0) circle (1cm);
\draw[vector]  (0.2,0) --  (45:1);
\draw[vector] (-0.2,0) -- (135:1);
\draw[vector] (-0.2,0) -- (225:1);
\draw[vector]  (0.2,0) -- (315:1);
\draw[adjoint] (-0.2,0) -- (0.2,0);
\end{tikzpicture}
\raisebox{0.9cm}{\,$-$\,}
\begin{tikzpicture}
\fill[fill=blue!30] (0,0) circle (1cm);
\draw[vector] (0,0.2) -- (45:1);
\draw[vector] (0,0.2) -- (135:1);
\draw[vector] (0,-0.2) -- (225:1);
\draw[vector] (0,-0.2) -- (315:1);
\draw[adjoint] (0,-0.2) -- (0,0.2);
\end{tikzpicture}
\raisebox{0.9cm}{\,$\Bigg)$}
\end{multline}
where the coefficient $c_0$ is given by
\begin{equation*}
[\alpha][\beta][\gamma]
\frac{[\alpha+\beta]}{[2\alpha+2\beta]}
\frac{[\alpha+\gamma]}{[2\alpha+2\gamma]}
\frac{[\beta+\gamma]}{[2\beta+2\gamma]}
\frac{[2\tau]}{[\tau]}
\frac{[\tau-t]}{[\tau+t]^2}
\end{equation*}
and the coefficient $c_1$ can be found by confluence.
\caption{Square relation}\label{fig:sqx}
\end{figure}

The skein relation for the unoriented, framed braid group generators is given in Figure~\ref{skeinx}.
\begin{figure}\begin{multline*}
	\begin{tikzpicture}[red,line width=2pt]
	\fill[blue!30] (0,0) circle (1cm);
	\draw (45:1) -- (225:1);
	\fill[blue!30] (0,0) circle (0.25cm);
	\draw (135:1) -- (315:1);
	\end{tikzpicture}
	\raisebox{0.9cm}{\,$-$\,}
	\begin{tikzpicture}[red,line width=2pt]
	\fill[blue!30] (0,0) circle (1cm);
	\draw (135:1) -- (315:1);
	\fill[blue!30] (0,0) circle (0.25cm);
	\draw (45:1) -- (225:1);
	\end{tikzpicture}
	\raisebox{0.9cm}{\,$=$} \\
	\raisebox{0.9cm}{$[\alpha][\beta][\gamma]\frac{[t]}{[2t]}\,\Bigg($\,}
	\begin{tikzpicture}[red,line width=2pt]
	\fill[blue!30] (0,0) circle (1cm);
	\draw (45:1) .. controls (45:0.5) and (315:0.5) .. (315:1);
	\draw (135:1) .. controls (135:0.5) and (225:0.5) .. (225:1);
	\end{tikzpicture}
	\raisebox{0.9cm}{\,$-$\,}
	\begin{tikzpicture}[red,line width=2pt]
	\fill[blue!30] (0,0) circle (1cm);
	\draw (45:1) .. controls (45:0.5) and (135:0.5) .. (135:1);
	\draw (315:1) .. controls (315:0.5) and (225:0.5) .. (225:1);
	\end{tikzpicture}
	\raisebox{0.9cm}{\,$\Bigg)$} \\
\raisebox{0.9cm}{$+\frac{[2\alpha+2\beta]}{[\alpha+\beta]}
	\frac{[2\alpha+2\gamma]}{[\alpha+\gamma]}
	\frac{[2\beta+2\gamma]}{[\beta+\gamma]}
	\frac{[\tau]}{[2\tau]}[\tau+t]$\,$\Bigg($\,}
\begin{tikzpicture}
\fill[fill=blue!30] (0,0) circle (1cm);
\draw[vector]  (0.2,0) --  (45:1);
\draw[vector] (-0.2,0) -- (135:1);
\draw[vector] (-0.2,0) -- (225:1);
\draw[vector]  (0.2,0) -- (315:1);
\draw[adjoint] (-0.2,0) -- (0.2,0);
\end{tikzpicture}
\raisebox{0.9cm}{\,$-$\,}
\begin{tikzpicture}
\fill[fill=blue!30] (0,0) circle (1cm);
\draw[vector] (0,0.2) -- (45:1);
\draw[vector] (0,0.2) -- (135:1);
\draw[vector] (0,-0.2) -- (225:1);
\draw[vector] (0,-0.2) -- (315:1);
\draw[adjoint] (0,-0.2) -- (0,0.2);
\end{tikzpicture}
\raisebox{0.9cm}{\,$\Bigg)$}
	\end{multline*}
	\caption{Skein relation}\label{skeinx}
\end{figure}

\subsection{Knot invariants}
The quantised extended Vogel plane is a ribbon category and so each object
gives rise to an invariant of framed links with values in $\End(I)$
by the coherence theorem in \cite{MR1268782}. 
In this section we show that this link invariant can be computed
for the closures of two string braids. 

The formula for the invariant of the closure of $\sigma^k$ is given in terms of the quantum dimensions and the eigenvalues in Table~\ref{table:brqx} by the following expression:
\begin{multline}
	(-1)^k\left(\dim_q\fg\, q^{2kt}+\dim_q X\right) \\
	+q^{2k(\tau+t)}+\dim_q(Y(\alpha))\, q^{2k\alpha}+\dim_q(Y(\beta))\, q^{2k\beta}
	+\dim_q(Y(\gamma))\, q^{2k\gamma} \\
\end{multline}
This invariant is multiplicative under disjoint union of links.
It is more convenient to divide by $\delta$ to get an invariant
of non-empty links which is multiplicative under connected sum.

This is an invariant of framed links. The invariant of oriented links
is given by multiplying by $q^{2(\tau+t)\mathrm{writhe}}$. The writhe of
the closure of $\sigma^k$ is $k$.

\begin{ex}
	For the Hopf link the invariant is given in Figure~\ref{fig:hopfx}.
\end{ex}
\begin{figure}
	\begin{multline*} 
		[\tau+t]
		\left([\tau+\alpha+\beta-\gamma]+[\tau+\alpha-\beta+\gamma]+[\tau-\alpha+\beta+\gamma]\right) \\
		-\frac{[\tau+\alpha+\beta][\tau+\alpha+\gamma][\tau+\beta+\gamma]}%
	{[\alpha][\beta][\gamma]}
		\frac{[2t]}{[t]}\frac{[\tau]}{[2\tau]}
	\end{multline*}
	\caption{Extended invariant of Hopf link}\label{fig:hopfx}
\end{figure}

\begin{ex}
	For the trefoil the invariant is the Laurent polynomial is given in Figure~\ref{fig:trefoilx}.
\end{ex}
\begin{figure}
	\begin{gather*} 
		-q^{4\tau+4\alpha+4\beta+4\gamma}	\\
		+q^{4\tau+4\alpha+4\beta+2\gamma}
		+q^{4\tau+4\alpha+2\beta+4\gamma}
		+q^{4\tau+2\alpha+4\beta+4\gamma} \\
		-q^{4\tau+4\alpha+2\beta+2\gamma}
		-q^{4\tau+2\alpha+4\beta+2\gamma}
		-q^{4\tau+2\alpha+2\beta+4\gamma} \\
		-q^{2\tau+4\alpha+4\beta+4\gamma}
		+q^{4\tau+2\alpha+2\beta+2\gamma} \\
		-q^{4\tau+2\alpha+2\beta}
		-q^{4\tau+2\alpha+2\gamma}
		-q^{4\tau+2\beta+2\gamma} \\
		+q^{2\tau+4\alpha+4\beta+2\gamma}
		+q^{2\tau+4\alpha+2\beta+4\gamma}
		+q^{2\tau+2\alpha+4\beta+4\gamma} \\		
		-q^{2\tau+4\alpha+2\beta+2\gamma}
		-q^{2\tau+2\alpha+4\beta+2\gamma}
		-q^{2\tau+2\alpha+2\beta+4\gamma} \\
		+q^{2\tau+4\alpha+2\beta}
		+q^{2\tau+4\alpha+2\gamma}
		+q^{2\tau+2\alpha+4\beta}
		+q^{2\tau+2\alpha+4\gamma}
		+q^{2\tau+4\beta+2\gamma}
		+q^{2\tau+2\beta+4\gamma} \\
		+4q^{2\tau+2\alpha+2\beta+2\gamma}
		-q^{4\tau} \\
		-q^{2\tau+2\alpha+2\beta}
		-q^{2\tau+2\alpha+2\gamma}
		-q^{2\tau+2\beta+2\gamma} \\
		+2q^{2\tau+2\alpha}
		+2q^{2\tau+2\beta}
		+2q^{2\tau+2\gamma} \\
		-q^{4\alpha+2\beta+2\gamma}
		-q^{2\alpha+4\beta+2\gamma}
		-q^{2\alpha+2\beta+4\gamma} \\		
		-q^{4\alpha}
		-q^{4\beta}
		-q^{4\gamma} \\
		-q^{2\tau}+q^{2\alpha+2\beta+2\gamma} \\
		-2q^{2\alpha+2\beta}
		-2q^{2\alpha+2\gamma}
		-2q^{2\beta+2\gamma} \\
		+q^{2\tau-2\alpha}
		+q^{2\tau-2\beta}
		+q^{2\tau-2\gamma} \\				
		+q^{2\alpha}
		+q^{2\beta}
		+q^{2\gamma} \\		
		+q^{2\alpha-2\beta}
		+q^{2\alpha-2\gamma}
		+q^{2\beta-2\gamma}
		+q^{-2\alpha+2\beta}
		+q^{-2\alpha+2\gamma}
		+q^{-2\beta+2\gamma} \\		
		-3+q^{2\tau-2\alpha-2\beta-2\gamma}\\
		+q^{-2\tau+2\alpha}
		+q^{-2\tau+2\beta}
		+q^{-2\tau+2\gamma} \\
		-q^{-2\alpha-2\beta}
		-q^{-2\alpha-2\gamma}
		-q^{-2\beta-2\gamma}\\	
		+q^{-2\tau-2\alpha-2\beta+2\gamma}
		+q^{-2\tau-2\alpha+2\beta-2\gamma}
		+q^{-2\tau+2\alpha-2\beta-2\gamma} \\		
		+q^{-2\tau-2\alpha}
		+q^{-2\tau-2\beta}
		+q^{-2\tau-2\gamma} \\
		+q^{-2\tau-2\alpha-2\beta-2\gamma}
	\end{gather*}
	\caption{Extended invariant of trefoil}\label{fig:trefoilx}
\end{figure}

\clearpage

\printbibliography
\end{document}